\documentclass[11pt]{amsart}
\usepackage{amssymb}
\usepackage{graphics,epsf,psfrag,epsfig}
\usepackage{latexsym}
\usepackage{amsmath}
\usepackage{amsfonts}
\usepackage{eufrak}
\usepackage{color}

\newtheorem{clm}{Claim}
\newtheorem{theorem}{Theorem}
\newtheorem{proposition}{Proposition}
\newtheorem{lemma}{Lemma}
\newtheorem{remark}{Remark}

\newcommand{\dir}{\text{Dir}}

\newcommand{\gd}{\delta}
\newcommand{\gep}{\epsilon}       

\newcommand{\hF}{\hat{F}}

\newcommand{\hX}{\hat{X}}

\newcommand{\tgt}{\tilde{\theta}}
\newcommand{\tgh}{\tilde{\eta}}
\newcommand{\tgo}{\tilde{\omega}}

\newcommand{\tk}{\tilde{k}}
\newcommand{\tr}{\tilde{r}}

\newcommand{\tD}{\tilde{D}}

\newcommand{\gG}{\Gamma}

\newcommand{\go}{\omega}

\newcommand{\gO}{\Omega}

\newcommand{\gt}{\theta}

\renewcommand{\ge}{\geq}
\renewcommand{\le}{\leq}
\newcommand{\RR}{\mathbb{R}}
\newcommand{\ZZ}{\mathbb{Z}}

\newcommand{\QQ}{\mathbb{Q}}

\newcommand{\mG}{\mathcal{G}}
\newcommand{\mH}{\mathcal{H}}
\newcommand{\mC}{\mathcal{C}}
\newcommand{\mF}{\mathcal{F}}

\newcommand{\mT}{\mathcal{T}}
\newcommand{\mE}{\mathcal{E}}

\newcommand{\mN}{\mathcal{N}}
\newcommand{\mV}{\mathcal{V}}

\newcommand{\mB}{\mathcal{B}}

\newcommand{\bP}{\mathbf{P}}
\newcommand{\bE}{\mathbf{E}}

\renewcommand{\tilde}{\widetilde}
\renewcommand{\hat}{\widehat}

\begin{document}

\title{Geodesics Toward Corners in First Passage Percolation}

\author[K. S. Alexander]{Kenneth S. Alexander}
\address{K.~S.~Alexander, Department of Mathematics, KAP 104\\
University of Southern California\\
Los Angeles, CA  90089-2532 USA}
\email{alexandr@usc.edu}

\author[Q. Berger]{Quentin Berger}
\address{Q. Berger, Sorbonne Universit\'e, LPSM\\
Campus Pierre et Marie Curie, case 158\\
4 place Jussieu, 75252 Paris Cedex 5, France}
\email{quentin.berger@sorbonne-universite.fr}

\maketitle

\begin{abstract}
For stationary first passage percolation in two dimensions, the existence and uniqueness of semi-infinite geodesics directed in particular directions or sectors has been considered by Damron and Hanson \cite{DH14}, Ahlberg and Hoffman \cite{AH16}, and others. However the main results do not cover geodesics in the direction of corners of the limit shape $\mathcal{B}$, where two facets meet.
We construct an example with the following properties: (i) the limiting shape is an octagon, (ii) semi-infinite geodesics exist only in the four axis directions, and (iii) in each axis direction there are multiple such geodesics.  Consequently, the set of points of $\partial \mathcal{B}$ which are in the direction of some geodesic does not have all of $\mathcal{B}$ as its convex hull.
\end{abstract}

\section{Introduction}
We consider stationary first passage percolation (FPP) on a lattice $\mathbb{L}$ with site set $\ZZ^2$, and with a set of bonds which we denote~$\mE$. We are mainly interested in the usual set of nearest-neighbor bonds $\mE = \{ (x,y) ;  \Vert x-y \Vert_1 =1\}$, though in Section~\ref{sec:simple} we consider $\mE$ with added diagonal bonds to construct a simpler example.
Each bond $e$ of $\mE$ is assigned a random passage time $\tau_e\geq 0$, and the configuration $\tau$ is assumed stationary under lattice translations; the measure on the space $\Omega = [0,\infty)^\mE$ of configurations $\tau$ is denoted $\bP$, with corresponding expectation $\bE$.  For sites $x,y$ of $\mathbb{L}$, a \emph{path} $\gamma$ from $x$ to $y$ in $\mathbb{L}$ is a sequence $x=x_0,\dots,x_n=y$ with $x_i,x_{i+1}$ adjacent in $\mathbb{L}$ for all $i$; we may equivalently view $\gamma$ as a sequence of edges.  The \emph{passage time} $T(\gamma)$ of $\gamma$ is $T(\gamma) = \sum_{e\in\gamma} \tau_e$. 
For sites $x,y$ we define
\[
  T(x,y) = \inf\{T(\gamma): \gamma \text{ is a path from $x$ to } y\}.
\]
A \emph{geodesic} from $x$ to $y$ is a path which achieves this infimum.  A \emph{semi-infinite geodesic} $\Gamma$ from a site $x$ is a path with (necessarily distinct) sites $x=x_0,x_1,\dots$ for which every finite segment is a geodesic, and the \emph{direction} of $\Gamma$, denoted $\dir(\Gamma)$, is the set of limit points of $\{x_n/|x_n|: n\geq 1\}$. It is of interest to understand semi-infinite geodesics, and in particular the set of directions in which they exist. 

It is standard to make the following assumptions, from \cite{Ho08}.

\medskip
\noindent {\bf Assumption A1.} 
\begin{itemize}
\item[(i)] $\bP$ is ergodic with respect to lattice translation;
\item[(ii)] $\bE(\tau_e^{2+\gep})<\infty$ for any $e\in\mE$, for some $\gep>0$.
\end{itemize}
\medskip

\noindent Under A1, an easy application of Kingman's sub-additive theorem gives that for each $x\in\ZZ^2$ the limit
\[
  \mu(x) = \lim_{n\to\infty} \frac{T(0,nx)}{n}
\]
exists.  This $\mu$ extends to $\QQ^2$ by restricting to $n$ for which $nx\in\ZZ^2$, and then to $\RR^2$ by continuity; the resulting function is a norm (provided the limit shape defined below is bounded). Its unit ball is a nonempty convex symmetric set which we denote $\mB$. The \emph{wet region} at time $t$ is $\mB(t) = \{x+[-\tfrac 12,\tfrac 12]^2: T(0,x)\leq t\}$.
The shape theorem of Boivin \cite{Bo90} says that with probability one, given $\gep>0$, for all sufficiently large $t$ we have 
\[
  (1-\gep)\mB \subset \frac{\mB(t)}{t} \subset (1+\gep)\mB,
\] 
so $\mB$ is called the \emph{limit shape}.
H\"aggstr\"om and Meester \cite{HM95} showed that every compact convex~B with the symmetries of $\mathbb Z^2$ arises as the limit shape for some stationary FPP process.

We add the following assumptions, also used in \cite{Ho08}, \cite{DH14} and outlined in \cite{AH16} (see {\bf A2}).

\medskip
\noindent {\bf Assumption A2.}
\begin{itemize}
\item[(iii)] $\bP$ has all the symmetries of the lattice $\mathbb{L}$;
\item[(iv)] if $\alpha,\gamma$ are finite paths, with the same endpoints, differing in at least one edge then $T(\alpha) \neq T(\gamma)$ a.s.;
\item[(v)] $\bP$ has \emph{upward finite energy}: for any bond $e$ and any $t$ such that $\bP(\tau_e>t)>0$, we have 
\[
  \bP(\tau_e>t\mid \{\tau_f:f\neq e\}) >0 \quad \text{a.s.};
\]
\item[(vi)] the limit shape $\mB$ is bounded (equivalently, $\mu$ is strictly positive except at the origin.) 
\end{itemize}
\medskip

Thanks to (iv), the union of all geodesics from a fixed site $x$ to sites $y\in\ZZ^2$ is a tree, and we denote it $\mT_x=\mT_x(\tau)$.  By \cite{Ho08}, $\mT_x$ contains at least 4 semi-infinite geodesics, and  Brito and Hoffman \cite{BH17} give an example in which there are only 4 geodesics, and the direction for each corresponds to an entire closed quadrant of the lattice.

To describe the directions in which semi-infinite geodesics may exist, we introduce some terminology. 
A \emph{facet} of $\mB$ is a maximal closed line segment $F$ of positive length contained in $\partial \mB$; the unique linear functional equal to~1 on $F$ is denoted $\rho_F$.  For each angle from~0 there corresponds a unique point of $\partial \mB$ in the ray from~0 at that angle; a facet thus corresponds to a sector of angles, or of unit vectors. We say a point $v\in\partial \mB$ is of \emph{type i} ($i=0,1,2$) if $v$ is an endpoint of $i$ facets.  
 We may divide points of $\partial \mB$ (or equivalently, all angles) into 6 classes:
\begin{itemize}
\item[(1)]  \emph{exposed points of differentiability}, that is, exposed points of $\partial \mB$ where $\partial \mB$ is differentiable, necessarily type~0;
\item[(2)] \emph{facet endpoints of differentiability}, or equivalently, type-1 points where $\partial \mB$ is differentiable;
\item[(3)]  \emph{facet interior points}, necessarily type-0;
\item[(4)] \emph{half rounded corners}, that is, type-1 points where $\partial \mB$ is not differentiable;
\item[(5)] \emph{fully rounded corners}, that is, type-0 points where $\partial \mB$ is not differentiable;
\item[(6)] \emph{true corners}, meaning type-2 points.
\end{itemize}
Associated to any semi-infinite geodesic $\Gamma=\{x_0,x_1,\dots\}$ is its \emph{Busemann function} $B_\Gamma:\ZZ^2\times\ZZ^2\to\RR$ given by
\[
  B_\Gamma(x,y) = \lim_{n\to\infty} (T(x,x_n) - T(y,x_n)).
\]
From \cite[Theorems 2.5 and 2.6]{AH16}, we know the following, under Assumptions A1 and A2. Almost surely, there exists for any semi-infinite geodesic $\Gamma$ a linear functional $\rho_\Gamma$ on $\RR^2$ with the property that $B_\Gamma$ is \emph{linear to} $\rho_\Gamma$, that is, 
\[
  \lim_{|x|\to\infty} \frac{1}{|x|}\left| B_\Gamma(0,x) - \rho_\Gamma(x) \right| = 0.
\]
Still from \cite{AH16}, the set $\{\rho_\Gamma=1\}$ is always a supporting line of $\mB$, so its intersection with $\partial \mB$ is either an exposed point $v$ or a facet $F$, and then $\dir(\Gamma)$ is equal to $\{v\}$ or contained in $F$ (modulo normalizing to unit vectors.)  Thus $\dir(\Gamma)$ determines $\rho_\Gamma$, unless $\dir(\Gamma)$ consists of only a corner of some type.  Furthermore, there is a closed set $\mC_*$ of linear functionals such that the set of functionals $\rho_\Gamma$ which appear for some $\Gamma$ is almost surely equal to $\mC_*$.  If $v\in\partial\mB$ is not a corner, there is a $\rho$ such that $\{\rho=1\}$ is the unique tangent line to $\partial\mB$ at $v$, and we have $\rho\in\mC_*$.

In  \cite{AH16}, Ahlberg and Hoffman define  a \emph{random coalescing geodesic} (or \emph{RC geodesic}), which is, in loose terms, a mapping which selects measurably for each $\tau$ a semi-infinite geodesic $\Gamma_0=\Gamma_0(\tau)$ in $\mT_0(\tau)$, in such a way that when the mapping is applied via translation to obtain $\Gamma_x\in\mT_x$, $\Gamma_0$ and $\Gamma_x$ coalesce a.s.\ for all $x$.  The following statements are valid under Assumptions A1 and A2: they are part of, or immediate consequences of, results of Ahlberg and Hoffman \cite[Theorems 12.1 and 12.2]{AH16}, strengthening earlier results from \cite{Ho08} and~\cite{DH14}. 
\begin{itemize}
\item[(I)] For each exposed point of differentiability $v\in\partial \mB$, there is a.s.~a unique RC geodesic $\Gamma$ with $\dir(\Gamma) = \{v/|v|\}$. 
\item[(II)] For each half rounded corner $v\in\partial \mB$, there is a.s.~at least one RC geodesic $\Gamma$ with $\dir(\Gamma) = \{v/|v|\}$; for one such $\Gamma$ the linear functional $\rho_\Gamma$ corresponds to a limit of supporting lines taken from the non-facet side of $v$.  This uses the fact that $\mC_*$ is closed. 
\item[(III)] For each fully rounded corner $v\in\partial \mB$, there are a.s.~at least two RC geodesics $\Gamma$ with $\dir(\Gamma) = \{v/|v|\}$, with distinct linear functionals $\rho_\Gamma$ corresponding to limits of supporting lines from each side of $v$. 
\item[(IV)] Given a facet $F$ with corresponding sector $S_F$ of unit vectors, there is a.s.~a unique RC geodesic $\Gamma$ with $\dir(\Gamma)\subset S_F$ and $\rho_\Gamma=\rho_F$.  For any other RC geodesic $\Gamma$ with $\dir(\Gamma)\cap S_F\neq \emptyset$, this intersection is a single endpoint of $S_F$ which must be a corner.
\end{itemize}

 But relatively little has been proved about geodesics, or RC geodesics, in the directions of true corners  (where several supporting lines coexist), for instance when the limit shape is a polygon, see e.g.\ the discussion in Section~3.1 of \cite{DH17}.  One may ask, must every true-corner direction be in $\dir(\Gamma)$ for some geodesic $\Gamma$?  Equivalently, must the convex hull of 
\[
  \mV_{\rm geo} := \{v\in\partial \mB: v/|v| \in\dir(\Gamma) \text{ for some semi-infinite geodesic } \Gamma\}
  \]
be all of $\mB$?
Further, we can consider the nonuniqueness set
\[
  \mN := \big\{ u\in S^1: \text{ there exist multiple semi-infinite geodesics $\Gamma$ with } \dir(\Gamma) = \{u\} \big \}.
\]
For each fully rounded corner $v$ we have $\bP(v\in\mN)=1$.
For each non-corner $v\in\partial \mB$ we have $\bP(v\in\mN)=0$, but in the case of $\mB$ with no corners this does not mean $\mN$ is empty.
 If every point of $\partial \mB$ is an exposed point of differentiability then there is at least one geodesic in every direction;
the union of all semi-infinite geodesics from 0 is therefore a tree with infinitely many branches, and each branching produces a point of $\mN$, so $\mN$ is infinite a.s.  In the example of Brito and Hoffman \cite{BH17}, the limit shape is a diamond with true corners on the axes, and for each of these corners $v$ there is a.s.~no geodesic $\Gamma$ with $\dir(\Gamma)= \{v/|v|\}$, so $\bP(v\in\mN)=0$.
This suggests the question, must $\bP(v\in\mN)=0$ for true corners? 

\smallskip
Our primary result is an example of FPP process, which we call \emph{fast diagonals FPP}, in which some true corners have no geodesic, and others have multiple geodesics a.s.\  This means in particular that the convex hull of $\mV_{\rm geo}$ is not all of $\mB$.

\begin{theorem}\label{main}
The fast diagonals FPP process (defined in Section \ref{sec:construct}) satisfies Assumptions~A1-A2, and 
has the following properties:
\begin{itemize}
\item[(i)] The limit shape is an octagon, with corners on the axes and main diagonals.
\item[(ii)] Almost surely, every semi-infinite geodesic $\Gamma$ is directed in an axis direction (that is, $\dir(\Gamma)$ consists of a single axis direction.)
\item[(iii)] Almost surely, for each axis direction there exist at least two semi-infinite geodesics directed in that direction.
\end{itemize}
\end{theorem}

We first introduce a simpler example in Section \ref{sec:simple}, which does not satisfy Assumption~A2 (in particular (iv) and (v)), but encapsulates the key ideas of our construction. The remaining main part of the paper is devoted to modifying this example in order to satisfy Assumption~A2,  which brings many complications, see Section \ref{sec:main}.

In Theorem \ref{main} the shape is an octagon, and therefore \cite{Ho08} tells that there must therefore be at least eight geodesics, one for each flat edge of the shape. In our example, there are no geodesics associated to the supporting lines that only touch the shape at the diagonal corners, and the geodesics associated to the flat pieces are asymptotically directed along the axes. 
However, questions remains regarding the geodesics directed along the axis---for example, are there only two geodesics in each axis direction? Put differently, are there geodesics associated to supporting lines that intersect the shape only at the corners on the axes? Our guess is that, in our example, there are indeed a.s.~only two geodesics in the direction of the axis, but the question remains as to whether it is possible to build a model with more than two geodesics  in a corner direction.

Our result can perhaps be adapted to produce more general polygons, with polygon vertex directions alternating between those having two or more geodesics and those having none, and with no other directions with geodesics.
In higher dimensions, the possibility of analogous examples is unclear.

\begin{remark}\rm
We make some informal comments here, without full proof, about the linear functionals associated to the geodesics in Theorem \ref{main}, and about geodesics vs.~RC geodesics.

Let us consider the collection $\mG_{E,x}$ of geodesics directed eastward from $x$, and the union $\mT_{E,x}$ of all such geodesics. 
For $x=0$, each such eastward geodesic $\Gamma$ has a height $h(\Gamma)$ of its final point of intersection with the vertical axis.  We claim that all semi-infinite geodesics contained in $\mT_{E,0}$ are in $\mG_{E,0}$, and $h(\Gamma)$ is bounded over $\Gamma \in \mG_{E,0}$, a.s.  In fact there can be no westward geodesic contained in $\mT_{E,0}$ because any eastward geodesic from $0$ passing through any point near the negative horizontal axis sufficiently far west from $0$ must cross a northward or southward geodesic on its way back eastward, contradicting uniqueness of point-to-point geodesics (i.e. Assumption A2(iv).)  Regarding northward geodesics contained in $\mT_{E,0}$, they are ruled out when we show in the proof of (ii) in Section \ref{sec:i-iii} that there is a.s. a random $R$ such that, roughly speaking, no eastward geodesic ``goes approximately northward for a distance greater than $R$ before turning eastward,'' and similarly for southward in place of northward. This also shows the boundedness of $h(\Gamma)$.

Now any limit of geodesics in $\mG_{E,0}$ must be contained in $\mT_{E,0}$, and it follows from our claim that any such limit is in $\mG_{E,0}$.  

It follows that among all eastward geodesics $\Gamma$ in $\mG_{E,0}$ with a given value of $h(\Gamma)$, there is a leftmost one, where ``leftmost'' is defined in terms of the path $\Gamma$ in the right half plane after the point $(0,h(\Gamma))$.
Hence boundedness of $h$ shows there is a leftmost geodesic $\Gamma_L$ overall in $\mG_{E,0}$.  It is then straightforward to show that $\rho_{\Gamma_L}$ must be the linear functional equal to 1 on the side of $\mB$ connecting the positive horizontal axis to the main diagonal.  By \cite[Theorem 2.6]{AH16}, $\Gamma_L$ is a.s.~the unique geodesic with corresponding linear functional $\rho_{\Gamma_L}$, so by \cite[Theorem 12.1]{AH16},  $\Gamma_L$, viewed now as a function of the configuration $\tau$ and initial point $x$, must be an RC geodesic.

Similar considerations apply symmetrically to other directions and to rightmost geodesics.  Therefore we can replace ``geodesics'' with ``RC geodesics'' in Theorem \ref{main}(iii).
\end{remark}

\section{Simple example: diagonal highways only}
\label{sec:simple}

For this section we consider the lattice $\mathbb{L}$ with site set $\ZZ^2$, with the set of bonds $\mE = \{ (x,y) ; \Vert x-y \Vert_{\infty}  =1 \}$, which adds diagonal bonds to the usual square lattice. We frequently identify bonds and path steps by map directions: either SW/NE or SE/NW for diagonal bonds, and N, NE, etc.~for steps.  By \emph{axis directions} we mean horizontal and vertical, or N, E, W, S, depending on the context.  For $a$ preceding $b$ in a path $\gamma$, we write $\gamma[a,b]$ for the segment of $\gamma$ from $a$ to $b$.

\smallskip
We assign all horizontal and vertical bonds passage time $1$.  (This makes the model in a sense degenerate, which is one of the aspects we modify in Section~\ref{sec:main}.)  Let $\tfrac 12 < \theta < 1$ and $(2\theta)^{-1}<\eta<1$. For diagonal bonds, for $k\geq 1$, a \emph{highway of class} $k$ consists of $2^k-1$ consecutive bonds, all oriented SW/NE or all SE/NW.  The collection of all highways of all classes is denoted $\mathbb{H}$, and a highway configuration, denoted $\omega$, is an element of $\{0,1\}^{\mathbb{H}}$. When a coordinate is 1 in $\omega$ we say the corresponding highway is \emph{present} in $\omega$.  To obtain a random highway configuration, for each of the two orientations we let southernmost points of class-$k$ highways occur at each $x\in\ZZ^2$ independently with probability $(\theta/2)^k$, for each $k\geq 1$.  Every diagonal bond is a highway of class 0. For each diagonal bond $e$ we have
\begin{equation}\label{classk}
  \bP(e \text{ is in a present class-$k$ highway}) = 1 - \left( 1 - \frac{\theta^k}{2^k} \right)^{2^k-1} \leq \theta^k, \quad k\geq 1,
\end{equation}
so with probability one, $e$ is in only finitely many present highways.  Thus for diagonal $e$ we can define $k(e)=\max\{k: e$ is in a present class-$k$ highway$\}$ if this set is nonempty, and $k(e)=0$ otherwise, and then define its passage time
\begin{equation}\label{omegae}
  \tau_e = \begin{cases} \sqrt{2}(1+\eta^{k(e)}) &\text{if } k(e)\geq 1;\\ 3 &\text{if } k(e)=0. \end{cases}
\end{equation}
For all horizontal and vertical bonds we define $\tau_e=1$.  
Note that the value 3 ensures non-highway diagonal bonds never appear in geodesics.  

Let $A_1,A_2$ denote the positive horizontal and vertical axes, respectively, each including 0.  Let
\[
  r_k = \sum_{j= k}^{\infty} \theta^j = \frac{\theta^{k}}{1-\theta},
\]
so that $\bP( e$ is in some present highway of class $\ge k)\le r_k$.
We fix $C$ and take $k_0$ large enough so
\begin{equation}\label{assumps}
  \eta^{k_0}<\frac{1}{32},\quad r_{k_0} \leq \frac 12, \quad 
    0.1 \cdot 2^k(\eta^{k-1}-\eta^{k}) > \frac{4C}{r_k} \quad \text{for all } k\geq k_0,
\end{equation}
the last being possible by our choice of $\eta$.

\begin{proposition}
The stationary first passage percolation process defined as above has the following properties:
\begin{itemize}
\item[(i)] The limit shape is an octagon, with corners on the axes and main diagonals.
\item[(ii)] The only infinite geodesics are vertical and horizontal lines (and the only geodesics starting at the origin are indeed the vertical and horizontal axes).
\end{itemize}
\end{proposition}

\begin{proof}
Let us first show that $\mB$ is an octagon with Euclidean distance $1$ in the axis and diagonal directions, and a facet in each of the 8 sectors of angle $\pi/4$ between an axis and a diagonal.
As a lower bound for the passage time from $(0,0)$ to $(a,b)$, we readily have that for $0\le b\le a$  (the other cases being treated symmetrically):
\[ \tau\big( (0,0),(a,b) \big) \ge \sqrt{2}  b + (a-b)\, . \]
For an upper bound, it follows from $\theta>1/2$ that the horizontal (or vertical) distance from the origin to the nearest diagonal highway $H$ connecting the postive horizontal axis to height $b$ is $o(b)$ a.s.~(see \eqref{Ik} below), so as $b\to\infty$
\[ \tau\big( (0,0),(a,b) \big) \le  (\sqrt{2} +o(1)  ) b + (a-b)\, .\]
This reflects the fact that one route from $(0,0)$ to $(a,b)$ is to follow the axis horizontally to a diagonal highway~$H$, then follow $H$ to the top or right side of the rectangle $[0,a]\times[0,b]$, then follow that side of the rectangle to $(a,b)$, provided $H$ intersects the rectangle.
The linearity of the asymptotic expression $\sqrt 2 b + (a-b)$ means that the limit shape is flat between any diagonal and an adjacent axis, while the asymptotic speed is 1 out any axis or diagonal, so the limit shape $\mB$ is an octogon. 
This proves item (i) and we focus on item (ii).

\smallskip
\noindent
{\it Step 1. Construction of a ``success'' event.}
For $x$ in the first quadrant $Q$, let $\Delta(x)$ denote the Euclidean distance in the SW direction from $x$ to $A_1\cup A_2$.  
Let $\hat G_k=\{x\in Q:\Delta(x) = (2^{k-1}+1)\sqrt{2}\}$, which is a translate of $A_1\cup A_2$. For $j> k$ define three random sets of highways:
\[
  \mG_{k,j} = \{\text{all present SW/NE highways of class $j$ intersecting both $A_1\cup A_2$ and $\hat G_k$}\},
\]
\[
  \mG_{k,j}' = \{\text{all present SW/NE highways of class $j$ crossing $A_1\cup A_2$ but not $\hat G_k$}\},
\]
\[
  \mG_{k,j}'' = \{\text{all present SW/NE highways of class $j$ crossing $\hat G_k$ but not $A_1\cup A_2$}\}.
\]
Note that these three sets are independent of each other, and intersections with each along any given line have density at most $\theta^j$, by \eqref{classk}. 
We also highlight that $\mG_{k,j}$ for $j\le k-1$ is empty, since highways of class $\le k-1$ are too short to connect $A_1\cup A_2$ and $\hat G_k$.
Again, using \eqref{classk}, intersections with $\mG_{k,k}$ have density (over sites) 
\begin{equation}\label{Gkkdensity}
  \bP(0\in \mG_{k,k}) \ge 1 - \left( 1 - \frac{\theta^k}{2^k} \right)^{2^k-1} \geq 1 - e^{-\theta^k/2} \geq \frac{\theta^{k}}{2}.
\end{equation}  
Here $0\in \mG_{k,k}$ is a shorthand notation for 0 being in a highway in $\mG_{k,k}$, and the last inequality holds provided that $k$ is sufficiently large.
Let
\[
\hat   H_{i,k} = \text{the highway intersecting $A_i$ closest to 0 among all in $\cup_{j \ge k} \mG_{k,j}$}, \quad i=1,2,
\]
and let $\hat U_k=(\hat X_{1,k},0)$ and $\hat V_k=(0,\hat X_{2,k})$ denote the corresponding intersection points in $A_1$ and $A_2$.
Let $\hat \Omega_k$ denote the \emph{open} region bounded by $A_1\cup A_2,\hat G_k, \hat H_{1,k}$, and $\hat H_{2,k}$, see Figure~\ref{fig1}.

\begin{figure}[htbp]
\begin{center}
\includegraphics[scale=0.8]{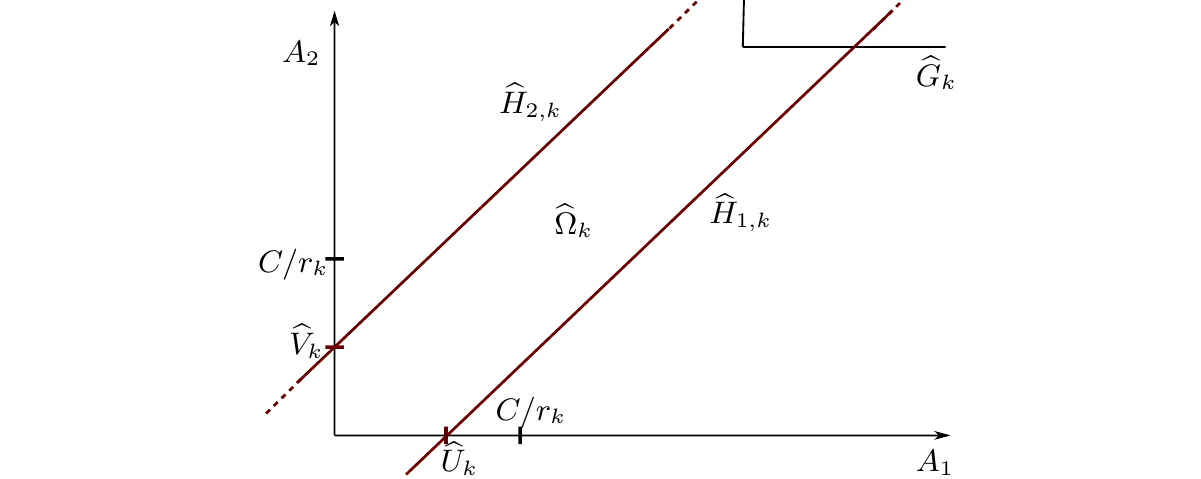}
\end{center}
\caption{\footnotesize 
Representation of the region $\hat\Omega_k$, which is enclosed by $A_1,A_2,\hat G_k$ and $\hat H_{1,k},\hat H_{2,k}$. Here, the event $\hat I_k$ is realized.}
\label{fig1}
\end{figure}

We define the event $\hat F_k = \hat I_k\cap \hat M_k$ (success at stage $k$) where
\begin{itemize}
\item[(i)] $\hat I_k: \max(\hat X_{1,k},\hat X_{2,k}) \leq C/r_k$, \quad with $C$ from \eqref{assumps}; 
\item[(ii)] $\hat M_k:$ every SW/NE highway intersecting $\hat \Omega_k$ is in classes $1,\dots,k-1$.
\end{itemize}
Note that a highway intersecting $\hat \Omega_k$ cannot intersect both $A_1\cup A_2$ and $\hat G_k$, by definition of~$\hat H_{i,k}$.

We claim there exists $\lambda>0$ such that 
\begin{equation}\label{claim}
  \text{$\bP(\hat F_k) \geq \lambda$ for all $k\geq k_0$.}
\end{equation}

Let us first prove that $\bP(\hat I_k)$ is bounded away from $0$.  Similarly to \eqref{Gkkdensity}, we have that  $\bP\big( 0\in \bigcup_{j\ge k} \mG_{k,j} \big)\ge r_k/3$,  so for $i=1,2$, by independence,
\begin{equation}\label{Ik}
  \bP(\hat X_{i,k} > C/r_k) \leq \exp\left( -C/3 \right) =: \zeta <1,
     \quad \text{for all } k\geq k_0.
\end{equation}
By independence, we get that $\bP(\hat I_k)\geq (1-\zeta)^2$.

We next prove that $\bP(\hat M_k\mid \hat I_k)$ is bounded away from 0 for $k\geq k_0$.  In fact, by the above-mentioned independence of the three sets of highways, by \eqref{classk} and \eqref{assumps} we have
\[
  \bP(\hat M_k\mid \hat I_k) \geq \min_{x_1,x_2\leq C/r_k} \bP( \hat M_k\mid \hat X_{1,k}=x_1,\hat X_{2,k}=x_2)
    \geq \min_{x_1,x_2\leq C/r_k} (1-r_k)^{x_1+x_2-1} \geq e^{-2C}.
\] 
This completes the proof of \eqref{claim}.  A slight modification of this proof shows that for fixed $\ell$, for $k$ sufficiently large we have $\bP(\hat F_k\mid \sigma(\hat F_1,\dots,\hat F_\ell))>\lambda/2$, and it follows that $\bP(\hat F_k\ i.o.)=1$.

\smallskip
\noindent
{\it Step 2. Properties of geodesics in case of a success.}
We now show that when $\hat F_k$ occurs, for every $x\notin A_1\cup A_2 \cup \hat  \Omega_k$ in the first quadrant, every geodesic $\hat \Gamma_{0x}$ from 0 to $x$ follows $A_1$ from 0 to $\hat U_k$, or $A_2$ from 0 to $\hat V_k$.  Since $x$ is in the first quadrant, it is easily seen that any geodesic from 0 to $x$ has only N, NE and E steps. 
Let $\hat p_x=(r,s)$ be the first site of $\hat \Gamma_{0x}$ not in $A_1\cup A_2 \cup \hat \Omega_k$.  Besides the geodesic $\hat \gG_{0x}[0, \hat p_x]$, we define an alternate path $\psi_x$ from 0 to $\hat p_x$ as follows; for this we assume $\hat p_x$ is in the first quadrant on or below the main diagonal, and make the definition symmetric for $\hat p_x$ elsewhere.
\begin{itemize}
\item[(i)] If $\hat p_x \in \hat H_{1,k}$ we let $\psi_x$ follow $A_1$ east from 0 to $\hat U_k$, then NE from $\hat U_k$ to $\hat p_x$ on $\hat H_{1,k}$.  
\item[(ii)] If $\hat p_x$ is in the horizontal part of $\hat G_k$ we let $\hat U_k'$ be the intersection of $\hat H_{1,k}$ with the vertical line through $\hat p_x$, and let $\psi_x$ be the path east from 0 to $\hat U_k$, then NE to $\hat U_k'$, then north to $\hat p_x$.  (Here we assume $k_0$ is chosen large enough so that $C/r_k < 2^{k-1}$, ensuring that $\hat p_x$ is farther east than $\hat U_k$.) 
\item[(iii)] Otherwise $\hat p_x$ is adjacent to $A_1$ and the final step of $\hat \gG_{0x}[0,\hat p_x]$ is from some $\hat U_k''\in A_1$ to $\hat p_x$, N or NE.  We let $\psi_x$ go east from 0 to $\hat U_k''$, then take one step (N or NE) to $\hat p_x$.
\end{itemize}

In case (i), the path $\psi_x$ has no N steps, and it is easily seen that any path in $A_1\cup A_2\cup\Omega_k$ from 0 to $\hat p_x$ containing some N steps will be strictly slower than $\psi_x$, and hence is not a geodesic.  Thus every geodesic from 0 to $\hat p_x$ has $s$ NE and $r-s$ E steps. 
Since success occurs at stage $k$, any diagonal bonds in $\hat \gG_{0x}[0,\hat p_x] \backslash \hat H_{1,k}$ have passage time strictly more than $\sqrt{2}(1+\eta^{k})$, making $\hat \gG_{0x}[0,\hat p_x]$ strictly slower than $\psi_x$, which contradicts the fact that $\hat \gG_{0x}[0,\hat p_x]$ is a geodesic.  It follows that we must have $\hat \gG_{0x}[0,\hat p_x] = \psi_x$, which means $\hat \Gamma_{0x}$ indeed follows $A_1$ from 0 to $\hat U_k$.  

In case (ii), we have $s=2^{k-1}+1$ and $0\leq r-s \leq\hat  X_{1,k}$.  Let $q$ be the number of NE steps in $\hat \gG_{0x}[0,\hat p_x]$, so it must have $s-q$ N and $r-q$ E steps.
Each of the diagonal bonds has passage time at least $\sqrt{2}(1+\eta^{k-1})$, so its passage time satisfies
\[
  T(\hat \gG_{0x}[0,\hat p_x]) \geq \sqrt{2}(1+\eta^{k-1})q + (s-q) + (r-q) \geq \sqrt{2}(1+\eta^{k-1})q + 2(s-q) \geq \sqrt{2}(1+\eta^{k-1})s.
\]
By contrast, the northward segment of $\psi_x$ has length $\hat X_{1,k}-(r-s)$, so
\begin{align*}
  T(\psi_x) &\leq 2\hat X_{1,k} - (r-s) + \sqrt{2}(1+\eta^{k})(r-\hat X_{1,k}) \\
  &\le 2 \hat X_{1,k} + \sqrt{2}(1+\eta^{k})s \\
  & \leq \sqrt{2}(1+\eta^{k})s + \frac{2 C}{r_k}\, .
\end{align*}
Hence by \eqref{assumps},
\[
  T(\hat \gG_{0x}[0,\hat p_x]) - T(\psi_x) \geq \sqrt{2}(\eta^{k-1}-\eta^{k})s - \frac{2C}{r_k} 
    = (2^{k-1}+1)\sqrt{2}(\eta^{k-1}-\eta^{k}) - \frac{2C}{r_k} > 0.
\]
But this contradicts the fact that $\hat \Gamma_{0x}$ is  a geodesic.  Thus we cannot have $\hat p_x$ in the horizontal part of $\hat G_k$ -- and similarly not in the vertical part.

In case (iii), since the unique geodesic between any two points of $A_1$ is a segment of $A_1$, it is straightforward that we must have $\hat \gG_{0x}[0,\hat p_x] = \psi_x$. Again, $\hat \Gamma_{0x}$  follows $A_1$  from $0$ to $\hat U_k$.


\smallskip
\noindent
{\it Step 3. Conclusion.}
 If $\hat \Gamma$ is any semi-infinite geodesic from 0 which has infinitely many points in the (closed) first quadrant, then for each of the infinitely many $k$ for which $\hat F_k$ occurs, the initial segment of $\hat \Gamma$ must follow an axis from the origin to $\hat U_k$ or $\hat V_k$.  But since $\hat X_{1,k},\hat X_{2,k} \to \infty$, this means $\hat \Gamma$ itself must be one of these axes. It follows that all horizontal and vertical lines are semi-infinite geodesics, and no other paths.  
\end{proof}

\section{Modification for square lattice, finite energy, and unique geodesics}
\label{sec:main}

The preceding simpler example does not satisfy A2 (iv) or (v), and it does not allow the use of results known only for the usual square lattice. 
The first difficulty is to adapt our construction of Section~2 so that it does not have diagonal bonds:
we remove the diagonal bonds, and we replace diagonal highways with \emph{zigzag highways} (alternating horizontal and vertical steps) as done in \cite{HM95}.
Then, in order to verify the upward finite energy, we need to introduce horizontal/vertical highways, and make highways of class $k$ not have a fixed length. In order to ensure  the unique geodesics condition, we add auxiliary randomization. All together, this adds significant complications.  Primarily, since the graph is planar, there is sharing of bonds between, for example, horizontal and zigzag highways where they cross. Since the passage times are different in the two types of highways, each shared bond slows or speeds the total passage time along at least one of the highways, compared to what it would be without the other highway.  We must ensure that the number of such crossings is not a primary determinant of which paths are geodesics.

Once we have properly define a passage time configuration (the \emph{fast diagonals} FPP), our strategy will be similar to that of the simple example of Section \ref{sec:simple}: we will define a ``success'' event, and show that when a success occurs geodesics approximately follow an axis for a long distance, at least until they reach a very fast  zigzag highway.

\subsection{Definition of the \emph{fast diagonals} FPP process}

 We now select parameters $\eta,\tgh,\theta,\tgt,\mu$: $\theta^k$ (resp.~$\tilde \theta^k$) will roughly correspond to the densities of zigzag (resp.~horizontal/vertical) highways of class $k$, defined below, and $\eta^k$ (resp.~$\tilde\eta^k$) will roughly correspond to the slowdown of a zigzag (resp.~horizontal/vertical) highway --- that is the higher the class, the faster the highway. We will choose the parameters so that the horizontal/vertical highways are not too infrequent relative to zigzag ones ($\tilde \theta^{c_\theta} > \theta$ for a certain power $c_\theta<1$)
and are less slowed down ($\tilde \eta <\eta$) than zigzag ones.  Loosely speaking we want the passage times to be much more affected by the zigzag highways encountered than by the horizontal/vertical ones, but fast zigzag highways have to be easily reachable by nearby horizontal/vertical highways, so the choice of parameters must be precise.

The actual choice of the parameters is as follows: we choose $c_{\theta} \in (0.4,0.5)$ and $c_{\tilde \theta} ,\gd >0$ to verify
\begin{equation}\label{params}
  \theta=2^{-c_\theta },\quad \tgt = 2^{-c_{\tgt}}, \quad \theta^{c_\theta c_{\tgt}/(1-c_\theta (c_{\tgt}-4\delta))} < \eta 
    < \min \Big(\frac 78,\theta^{2/3} \Big),\quad \tgh < \min\Big( \frac{1}{2\theta}, \eta/2\Big) , 
\end{equation}
and
\begin{equation}\label{params2}
    \theta < \mu < \min(\tgt^{c_\theta },\eta,(\theta\eta)^{1-c_\theta (c_{\tgt}-4\delta)}\theta^{-4c_\theta \delta}).
\end{equation}
To see that this choice of parameters is possible, notice that $c_\theta /(1-c_\theta )>2/3$ and $\theta^{c_\theta /(1-c_\theta )} < 7/8$ so we can choose $0<4\delta<c_{\tgt}<1$ such that $c_\theta c_{\tgt}/(1-c_\theta (c_{\tgt}-4\delta))>2/3$ and $\theta^{c_\theta c_{\tgt}/(1-c_\theta (c_{\tgt}-4\delta))} < 7/8$.
Note that since $c_\theta c_{\tgt}<0.5$, the third condition in \eqref{params} guarantees $\eta>\theta$; since $c_\theta<0.5$ this also means
\begin{equation}\label{etatheta}
  2\eta\theta >2\theta^2>1,
\end{equation}
 guaranteeing that one can make a modification of \eqref{assumps} hold also here:
\begin{equation}\label{assumps2}
  0.1 \cdot 2^k(\eta^{k-1}-\eta^{k}) > \frac{4C}{r_k} + 0.2 \quad \text{for all } k\geq k_0
\end{equation} 
by choosing $k_0$ large enough. 
Further, the first inequality in that third condition in \eqref{params} is equivalent to $\theta < (\theta\eta)^{1-c_\theta (c_{\tgt}-4\delta)}\theta^{-4c_\theta \delta}$.  Together these show $\mu$ can be chosen to satisfy~\eqref{params2}.

\smallskip
{\it Highways and types of bonds.}
A \emph{zigzag highway} is a set of (adjacent) bonds in any finite path which either (i) alternates between N and E steps, starting with either, called a SW/NE highway, or (ii) alternates between N and W steps, starting with either, called a SE/NW highway.  If the first step in the path is N, we say the highway is \emph{V-start}; if the first step is W or E we say it is \emph{H-start}. A SW/NE highway is called \emph{upper} if it is above the main diagonal, and \emph{lower} if it is below, and analogously for SE/NW highways.
Note a SW/NE highway is not oriented toward SW or NE, it is only a set of bonds, and similarly for SE/NW.  The \emph{length} $|H|$ of a highway $H$ is the number of bonds it contains.  To each zigzag highway $H$ we associate a random variable $\mathcal{U}_H$ uniformly distributed in $[0,1]$ and independent from highway to highway.

For each $k\geq 1$ we construct a random configuration $\omega^{(k)}$ of zigzag highways of class $k$: these highways can have any length $1,\dots,2^{k+3}$. Formally we can view $\omega^{(k)}$ as an element of $\{0,1\}^{\mathbb{H}_k}$, where $\mathbb{H}_k$ is the set of all possible class-$k$ zigzag highways; when a coordinate is 1 in $\omega^{(k)}$ we say the corresponding highway is \emph{present} in $\omega^{(k)}$.  
To specify the distribution, for each length $j\leq 2^{k+3}$ and each $x\in\ZZ^2$, a SW-most endpoint of a present length-$j$ $H$-start SW/NE highway of class $k$ occurs at $x$ with probability $\theta^k/2^{2k+4}$, independently over sites $x$ and classes $k$, with the same for $V$-start, and similarly for SE/NW highways.  We write $\omega=(\omega^{(1)},\omega^{(2)},\dots)$ for the configuration of zigzag highways of all classes. Note that due to independence, for $k< l$, a given zigzag highway (of length at most $2^{k+3}$) may be present when viewed as a class-$k$ highway, and either present or not present when viewed as a class-$l$ highway, and vice versa. Formally, then, a present highway in a configuration $\omega$ is an ordered pair $(H,k)$, with $k$ specifying a class in which $H$ is present, but we simply refer to $H$ when confusion is unlikely.

A \emph{horizontal highway} of length $j$ is a collection of $j$ consecutive horizontal bonds, and similarly for a \emph{vertical highway}.  Highways of both these types are called \emph{HV highways}.
For each $k\geq 1$ we construct a configuration $\tgo^{(k)}$ of HV highways of class $k$: these highways can have any length $1,\dots,2^k$, and for each length $j\leq 2^k$ and each $x\in\ZZ^2$, a leftmost endpoint of a present length-$j$ horizontal highway of class $k$ occurs at $x$ with probability $\tgt^k/2^{2k}$, independently over sites $x$, and similarly for vertical highways.  We write $\tgo = (\tgo^{(1)},\tgo^{(2)},\dots)$ for the configuration of HV highways of all classes.

\smallskip
We now combine the classes of zigzag highways and ``thin'' them into a single configuration by deletions ---we stress that following our definitions, each site is a.s.\ in only finitely many highways.  We do the thinning in two stages, first removing those which are too close to certain other zigzag highways, then those which are crossed by a  HV highway with sufficiently high class.

Specifically, for \emph{stage-1 deletions} we define a linear ordering (a \emph{ranking}) of the SW/NE highways present in $\omega$, as follows. Highway $(H',k)$ ranks above highway $(H,l)$ if one of the following holds:
(i) $k>l$;
(ii) $k=l$ and $|H'|>|H|$;
(iii)  $k=l$, $|H'|=|H|$, and $\mathcal{U}_{H'}>\mathcal{U}_H$.  
Let $d_1(A,B)$ denote the $\ell^1$ distance between the sets $A$ and $B$ of sites or bonds. 
We then delete any SW/NE highway $(H,k)$ from any $\omega^{(k)}$ if there exists another present SW/NE highway $(H',l)$, with $d_1(H,H')\leq 22$ which ranks higher than $(H,k)$.  
We then do the same for SE/NW highways.  
The configuration of highways that remain in some $\omega^{(k)}$ after stage-1 deletions is denoted $\omega^{\rm zig,thin,1}$. 
Here the condition $d_1(H,H')\leq 22$ is chosen to follow from $d_1(H,H')<3/(1-\eta)-1$; we have chosen $\eta<7/8$ in \eqref{params} so the value 22 works.

For \emph{stage-2 deletions}, we let $\zeta = \delta/(c_{\tgt}-\delta)$, and for each $m\geq 1$ we delete from $\omega^{\rm zig,thin,1}$ each zigzag highway of class $m$ which shares a bond with an HV highway of class $(1+\zeta)m/c_{\tgt}$ or more (as is always the case when such highways intersect, unless the intersection consists of a single endpoint of one of the highways.)  The configuration of highways that remain in some $\omega^{(k)}$ after both stage-1 and stage-2 deletions is denoted $\omega^{\rm zig,thin,2}$.  For a given zigzag highway $(H,m)$ in $\omega^{\rm zig,thin,1}$, for each bond $e$ of $H$ there are at most $2^\ell$ possible lengths and $2^\ell$ possible endpoint locations for a class-$\ell$ HV highway containing $e$, so
the probability $(H,m)$ is deleted in stage 2 is at most 
\begin{equation}\label{stage2prob}
  2^{m+3} \sum_{\ell\geq (1+\zeta)m/c_{\tgt}} 2^{2\ell}\cdot \frac{\tgt^\ell}{2^{2\ell}} = \frac{8}{1-\tgt}\ 2^{-\zeta m}.
\end{equation}

Following stage-2 deletions we make one further modification, which we call \emph{stage-3 trimming}. Suppose $(H,k),(H',l)$ are zigzag highways of opposite orientation (SW/NE vs SE/NW) in $\omega^{\rm zig,thin,2}$, and $x$ is an endpoint of $H$.  If $d_1(x,H')\leq 1$, then we delete from $H$ the 4 final bonds of $H$, ending at $x$, creating a shortened highway $\hat H$.  (Formally this means we delete $(H,k)$ from $\omega^{\rm zig,thin,2}$ and make $(\hat H,k)$ present, if it isn't already.)  The resulting configuration is denoted $\omega^{\rm zig,thin}$. This ensures that for any present SW/NE zigzag highway $H$ and SE/NW zigzag highway $H'$, either $H$ and $H'$ \emph{fully cross} (meaning they intersect, and there are at least 2 bonds of each highway on either side of the intersection bond) or they satisfy $d_1(H,H')\geq 2$.

\smallskip
This construction creates several types of bonds, which will have different definitions for their passage times.
A bond $e$ which is in no highway in $\omega^{\rm zig,thin}$ but which has at least one endpoint in some highway in $\omega^{\rm zig,thin}$ is called a \emph{boundary bond}.  A bond $e$ in any highway in $\omega^{\rm zig,thin}$ is called a \emph{zigzag bond}.
A \emph{HV bond} is a bond in some HV highway in some $\tgo^{(k)}$.  An \emph{HV-only bond} is an HV bond which is not a zigzag bond.  
A \emph{backroad bond} is a bond which is not a zigzag bond, HV-only bond, or boundary bond.

Moreover, there are special types of boundary and zigzag bonds  that arise when a SW/NE zigzag highway crosses a SE/NW one, so we need the following definitions for bonds in $\omega^{\rm zig,thin}$. For zigzag bonds, we distinguish:
\begin{itemize}
\item[(i)] The first and last bonds (or sites) of any zigzag highway are called \emph{terminal bonds} (or \emph{terminal sites}.)  A bond which is by itself a length-1 highway is called a \emph{doubly terminal bond}; other terminal bonds are \emph{singly terminal bonds}.  Bonds which are not the first or last bond of a specified path are called \emph{interior bonds}.
\item[(ii)] An adjacent pair of zigzag bonds in the same direction (both N/S or both E/W, which are necessarily from different highways, one SW/NE and one SE/NW) is called a \emph{meeting pair}, and each bond in the pair is a \emph{meeting zigzag bond}.  
\item[(iii)] A zigzag bond for which both endpoints are meeting-pair midpoints is called an \emph{intersection zigzag bond}. Equivalently, when a SW/NE highway intersects a SE/NW one, the bond forming the intersection is an intersection zigzag bond. 
\item[(v)] A zigzag bond which is not a meeting, intersection, or terminal zigzag bond is called a \emph{normal zigzag bond}.
\end{itemize}
For boundary bonds, we distinguish the following:
\begin{itemize}
\item[(vi)] A boundary bond is called a \emph{semislow boundary bond} if either (a) it is adjacent to two meeting bonds (and is necessarily parallel to the intersection bond, separated by distance 1), called an \emph{entry/exit bond}, or (b) it is adjacent to an intersection bond. 
\item[(vi)] A boundary bond $e$ is called a \emph{skimming boundary bond} if one endpoint is a terminal site of a zigzag highway, and the corresponding terminal bond is perpendicular to $e$.
\item[(vii)] A boundary bond which is not a semislow or skimming boundary bond is called a \emph{normal boundary bond}.
\end{itemize}
These special type of bonds are represented in Figure \ref{fig:crossing}.
\begin{figure}[htbp]
\begin{center}
\includegraphics[scale=1]{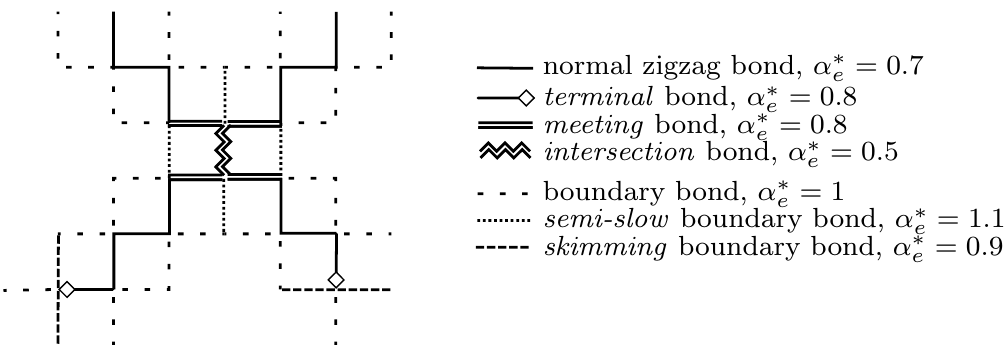}
\end{center}
\caption{\footnotesize Different types of bonds near a crossing of two zigzag highways, and their compensated core passage times, see \eqref{compcore}.
}
\label{fig:crossing}
\end{figure}

\smallskip
{\it Definition of the passage times.}
For each edge $e$ we associate its zigzag class 
\begin{equation*}
  k(e)= \begin{cases} \max\{k: e \text{ is in a class-$k$ zigzag highway in }\omega^{\rm zig, thin} \} &\text{if this set is nonempty}, \\ 0, &\text{otherwise}, \end{cases}
\end{equation*}
and its HV class
\begin{equation*}
  \tk(e)= \begin{cases} \max\{k: e \text{ is in a class-$k$ HV highway in }\tilde \omega^{(k)} \} &\text{if this set is nonempty}, \\ 0, &\text{otherwise}. \end{cases}
\end{equation*}
We then define a bond $e$ to be \emph{slow} if 
\begin{equation}\label{slowdef}
  \xi_e' \leq 4^{-k(e)\vee \tk(e)},
\end{equation}
where $\xi_e'$ is uniform in $[0,1]$, independent from bond to bond.
Slow bonds exist only to ensure that upward finite energy holds.

We next define the \emph{raw core passage time} $\alpha_e=\alpha_e(\tgo,\omega^{\rm zig,thin})$ of each bond $e$, mimicking ideas of the simple example of Section \ref{sec:simple}.  For non-slow $e$ we set
\begin{equation}\label{uncomp}
  \alpha_e = \begin{cases} 
    0.7 &\text{if  $e$ is a zigzag bond}\\
    0.9 &\text{if  $e$ is an HV-only bond}\\
    1 &\text{if $e$ is a backroad bond or non-HV boundary bond},
    \end{cases}
\end{equation}
and for slow $e$, we set $\alpha_e=1.2$.
Then we define the \emph{compensated core passage time} $\alpha_e^*=\alpha_e^*(\omega,\omega^{\rm zig,thin})$ for non-slow $e$ by
\begin{equation}\label{compcore}
  \alpha_e^* = \begin{cases} 
    0.5 &\text{if $e$ is an intersection zigzag bond}\\
    0.7 &\text{if  $e$ is a normal zigzag bond}\\
    0.8 &\text{if $e$ is a meeting bond, or a singly terminal zigzag bond}\\
    0.9 &\text{if  $e$ is  a doubly terminal zigzag bond, or an HV-only bond which is either }\\
    &\qquad \text{non-boundary or a skimming boundary bond}  \\
    1 &\text{if $e$ is a backroad bond, or  a normal boundary bond}\\
    1.1 &\text{if $e$ is a semislow boundary bond},
    \end{cases}  
\end{equation}
and for slow $e$, we set  $\alpha_e^*=1.3$. 
The term ``compensated'' refers mainly to the following:  when an HV highway crosses a zigzag one, it typically intersects one normal zigzag bond and two normal boundary bonds.  The sum of the raw core passage times for these 3 bonds is $0.9+0.7+0.9=2.5$, whereas in the absence of the zigzag highway the sum would be $3\times0.9=2.7$.  With the compensated times, the sum is restored to 2.7, and in that sense the HV highway does not ``feel'' the zigzag highway. 
The compensation picture is more complicated when the HV highway crosses near the intersection of a SW/NE highway and a SE/NW highway (which necessarily fully cross.) It must be done so that \eqref{alphasumA} and \eqref{alphasum} below hold, whether the HV highway contains intersection bonds, meeting bonds or entry/exit bonds, see Figure~\ref{fig:crossing}.

In a similar sense, a SW/NE highway does not ``feel'' a crossing by a SE/NW highway. 


The idea in the definition of $\alpha_e^*$ is that we compensate for the ``too fast'' zigzag bonds in an HV highway (0.7 vs 0.9) by extracting a toll of 0.1 for entering or exiting a zigzag highway.  If the entry/exit is through a terminal or meeting zigzag bond (as when passing through a meeting block), then the toll is paid by increasing the time of that bond from 0.7 to 0.8.  In the meeting case, to avoid increasing the total time along the zigzag highway, the core passage time of the adjacent intersection bond is reduced to 0.5.  If the entrance/exit for the zigzag highway is made through any other type of zigzag bond, the toll is paid by increasing the adjacent boundary bond in the path from 0.9 (if it's an HV bond) to 1.  There is an exception in an entry/exit bond, which may be both entrance and exit: the toll for such a bond is 0.2.

Due to the stage-1 deletions, any two parallel zigzag highways $H,H'$ in $\omega^{\rm zig,thin}$ satisfy 
\begin{equation}\label{six}
  d_1(H,H')\geq 23,
\end{equation}
 and as a result, the compensation described by \eqref{compcore} is sufficient for our purposes; without a lower bound like \eqref{six}, more complicated highway-crossing situations would be possible, producing for example meeting zigzag bonds which are also intersection zigzag bonds.  

We thus have the following property:  suppose $H$ is an HV highway for which the first and last bonds are HV-only, and $\Gamma$ is a path which starts and ends with non-zigzag bonds.  Then 
\begin{equation}\label{alphasumA}
  \sum_{e\in H} \alpha_e^* = 0.9|H|, \qquad \sum_{e\in\Gamma} \alpha_e^* \geq 0.7|\Gamma|.
\end{equation}
End effects may alter this for general HV highways and paths, but it is easily checked that every HV highway $H$ and path $\Gamma$ satisfies
\begin{equation}\label{alphasum}
  \Big| \sum_{e\in H} \alpha_e^* - 0.9|H| \Big| \leq 0.4, \qquad \sum_{e\in\Gamma} \alpha_e^* \geq 0.7|\Gamma| - 0.2.
\end{equation} 

We use another auxiliary randomization to ensure unique geodesics:  for each bond $e$ we let $\xi_e$ be uniform in $[0,1]$, independent from bond to bond (and independent of the $\xi_e'$'s used above).
We can now define the full passage times $\tau_e$ (based on the configurations $\omega^{\rm zig,thin}$ and $\tilde \omega$) by 
\begin{equation}\label{taudef}
  \tau_e= \alpha_e^* + \sigma_e,
\end{equation}
where $\sigma_e$ (the slowdown) is defined to be $0.1 \xi_e$ if $e$ is a slow bond, and for non-slow $e$ 
\begin{equation}\label{defsigmae}
\sigma_e = 0.1 \times
\begin{cases}
\eta^{k(e)} + \tgh^{k(e)}\xi_e &\text{if $e$ is a zigzag bond,} \\
\tgh^{\tk(e)} + \tgh^{\tk(e)}\xi_e &\text{if $e$ is an HV-only bond,}\\
\tgh^{\tk(e)}\xi_e &\text{if $e$ is either a backroad bond, or a boundary bond which is not HV.}
\end{cases}
\end{equation} 
Hence, this corresponds to a slowdown of order $\eta^k$ (per bond) in class-$k$ zigzag higways, and of order $\tgh^{\tilde k}$ (per bond) in class-$\tilde k$ HV-highways.
We refer to the resulting stationary FPP process as the \emph{fast diagonals FPP process}. 
We stress that the presence of the independent variables $\xi_e$ ensures that Assumption A2(iv) is satisfied, and (since $\sigma_e\leq 0.2$ in all cases) the presence of slow bonds ensures the positive finite energy condition  A2(v); the rest of A1 and A2 are straightforward.

\smallskip
{\it First observations and notations.}
It is important that since $\eta>\tgh$, passage times along long zigzag highways are much more affected by the class of the highway than are times along HV highways.   In fact, by increasing $k_0$ we may assume $(2\tgh)^{k_0}<\eta^{k_0}<1/16$, recalling we chose $\tilde \eta<\eta/2$. Then
if $H$ is an HV highway of class $k\geq k_0$  (so of length at most $2^{k+3}$), we have
\begin{equation}\label{tenthbound}
 0.1 \sum_{e\in H} (\tgh^{k} + \tgh^{k}\xi_e) \leq 1.6 (2\tgh)^k < 0.1,
\end{equation}  
so the maximum effect on $\sum_{e\in H} \tau_e$ of all the variables involving $\tgh$ is less than 0.1, hence is less than the effect of the $\alpha_e^*$ value for any single bond $e$.

\smallskip
Henceforth we consider only $k\geq k_0$, n setting the scale of $\hat\Omega_k$. Let $r_k = \sum_{j=k}^\infty \gt^j$ and $\tr_k = \sum_{j=k}^\infty \tgt^j$.  
Define $q_k = \log_2(2C/r_k) = c_\theta k+b$ where $b$ is a constant.

The following subsections prove Theorem \ref{main}.  The strategy is similar to that of Section \ref{sec:simple}: we first construct in Section \ref{sec:construct} an event $ F_k$ that a.s.\ occurs for a positive fraction of all $k$'s, and then show in Section \ref{sec:properties} that when $ F_k$ occurs geodesics have to stay near the axis. We conclude the proof of Theorem~\ref{main} in Section~\ref{sec:i-iii}.

\subsection{Construction of a ``success'' event}
\label{sec:construct}
Analogously to Section \ref{sec:simple}, we construct a random region $\gO_k$ (which is an enlarged random version of $\hat \Omega_k$), and a deterministic region $\Theta_k$ which may contain $\gO_k$, as follows.

 As before we write $A_1,A_2$ for the positive horizontal and vertical axes, $Q$ for the first quadrant, and now also $A_3, A_4$ for the negative horizontal and vertical axes, respectively.  We write $G_k^1$ for the set $\{x\in Q:\Delta(x) = 2^k\sqrt{2}\}$ (formerly denoted $\hat G_k$)  and $\tilde G_k^1$ for $\{x\in Q:\Delta(x) = (2^k+4)\sqrt{2}\}$; successively rotating the lattice by 90 degrees yields corresponding sets $G_k^2,G_k^3,G_k^4$ and $\tilde G_k^2,\tilde G_k^3,\tilde G_k^4$ in the second, third and fourth quadrants, respectively.
Let $H_{NE,L,k}$ and $H_{NE,U,k}$ be the lower and upper zigzag highways in $\omega^{\rm zig,thin,2}$ of class $k$ or more, intersecting both $\tilde G_k^1$ and $\tilde G_k^3$, which intersect $A_1$ and $A_2$, respectively, closest to $0$.  Let $U_k^{NE} = (X_{1,k}^{NE},0)$ be the leftmost point of $A_1\cap H_{NE,L,k}$, and $V_k^{NE}=(0,X_{2,k}^{NE})$ the lowest point of $A_2\cap H_{NE,U,k}$.  Rotating the lattice 90 degrees yields analogous highways $ H_{NW,L,k}$ and $ H_{NW,U,k}$ each intersecting $\tilde G_k^2$ and $ \tilde G_k^4$, and intersections points $U_k^{NW} = (X_{1,k}^{NW},0)$ and $V_k^{NW}=(0,X_{2,k}^{NW})$ with axes $A_3$ and $A_2$, respectively.  Here we have used $\tilde G_k^i$ and not $G_k^i$ so that stage-3 trimming does not prevent $H_{*,\cdot,k}$ from reaching appropriate $G_k^i$.


Let $\ell_{1,k}$ and $\ell_{3,k}$ be the vertical lines $\{\pm 2C/r_k\}\times\RR$ crossing $A_1$ and $A_3$ respectively, and $\ell_{2,k}$ and $\ell_{4,k}$ the horizontal lines $\RR\times\{\pm 2C/r_k\}$ crossing $A_2$ and $A_4$.
Let $J_{N,k}$ (and $J_{S,k}$, respectively) denote the lowest (and highest) horizontal highway above (and below) the horizontal axis intersecting both $\ell_{1,k}$ and $\ell_{3,k}$. Analogously, let $J_{E,k}$ (and $J_{W,k}$) be the leftmost (rightmost) vertical highway to the right (left) of the vertical axis  intersecting both $\ell_{2,k}$ and $\ell_{4,k}$.  Let $(Y_{E,k},0)$ be the intersection of $J_{E,k}$ with $A_1$, and analogously for $(0,Y_{N,k}),(Y_{W,k},0)$ and $(0,Y_{S,k})$ in $A_2,A_3$, and $A_4$.

Let $\gO_k^{NE}$ be the open region bounded by $H_{NE,L,k},H_{NE,U,k},G_k^1$, and $G_k^3$, and let $\gO_k^{NW}$ be the open region bounded by $H_{NW,L,k},H_{NW,U,k},G_k^2$, and $G_k^4$.  Then let $\gO_k=\gO_k^{NW}\cup\gO_k^{NE}$ (an $X$-shaped region), see Figure \ref{fig:Regions} below.

Let $h_{NE,L,k}$ and $h_{NE,U,k}$ denote the SW/NE diagonal lines through $(C/r_k,0)$ and $(0,C/r_k)$ respectively, and let $h_{NW,L,k}$ and let $h_{NW,U,k}$ denote the SE/NW diagonal lines through $(-C/r_k,0)$ and $(0,C/r_k)$, respectively.
Let $\Theta_k^{NE}$ denote the closed region bounded by $h_{NE,L,k},h_{NE,U,k},G_k^1$ and $G_k^3$, and $\Theta_k^{NW}$ the closed region bounded by $h_{NW,L,k},h_{NW,U,k},G_k^2$ and $G_k^4$.  Then let $\Theta_k=\Theta_k^{NW}\cup\Theta_k^{NE}$.  In the event of interest to us, the nonrandom region $\Theta_k$ will contain the random region $ \Omega_k$.

\begin{figure}[htbp]
\begin{center}
\includegraphics[scale=.8]{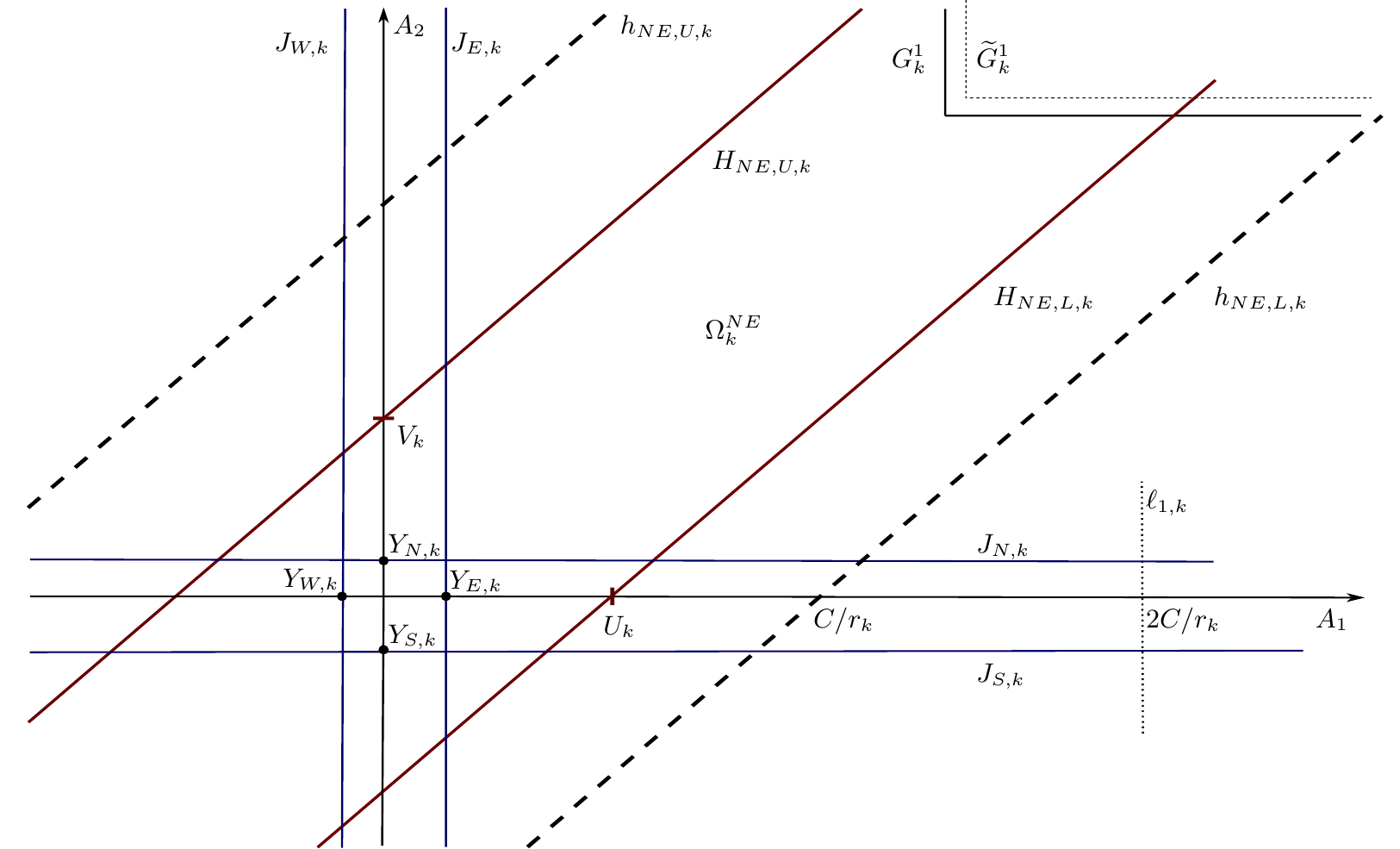}
\end{center}
\caption{\footnotesize Representation of the different regions of interest (the picture focuses on the upper-right quadrant, and the scales are not respected): $ \Omega_k^{NE}$ is the shaded region enclosed by $  H_{NE,U,k},  H_{NE,L,k}$ and $ G_k^i$ ($i=1,3$);   $\Theta_k^{NE}$ is the region enclosed by $  h_{NE,U,k},  h_{NE,L,k}$ and $ G_k^i$ ($i=1,3$). The construction of $ \Omega_k^{NW}$ and $\Theta_k^{NW}$ is symmetric by rotation of 90$^\circ$.}
\label{fig:Regions}
\end{figure}

\medskip
We now construct an event $F_k$, which we will show occurs for a positive fraction of all $k$, a.s.  We show in Section \ref{sec:properties} that $F_k$ ensures that all finite geodesics from 0 to points outside $\gO_k$ must stay near one of the axes until leaving $\gO_k^{NW}\cap\gO_k^{NE}$. 
For $m>c_{\tgt}q_k/(1+\zeta)$ let $R_{k,m}^*$ be the number of class-$m$ zigzag highways in $\omega^{(m)}$ intersecting $J_{*,k}$ in $\Theta_k$, for $*=$ N, E, S, W.  Here $\zeta$ is from the definition of stage-2 deletions. Note that since the class of $J_{*,k}$ is at least $q_k$, any intersecting highways  of class $m\leq c_{\tgt}q_k/(1+\zeta)$ are removed in stage-2 deletions, so $R_{k,m}^*$ counts those which might remain (depending on the class of $J_{*,k}$).

  Fix $c\in (0,C/2)$ and define the events
\[
   I_k^*:=\Big\{ \,  \frac{c}{r_k} \leq X_{i,k}^* \leq \frac{C}{r_k} \text{ for $i=1,2$}\, \Big\}, \text{ for $*=$ NE, NW}, \qquad I_k = I_k^{NE}\cap I_k^{NW},
\]
\begin{align*}
  M_k^{NE}:= 
  \bigg\{
  \begin{minipage}{8.5cm}
every SW/NE highway $H\notin \{H_{NE,L,k},H_{NE,U,k}\}$ in $\omega$ intersecting $\Theta_k^{NE}$  is in classes $1,\dots,k-1$
  \end{minipage}
  \bigg\}\, ,
  \quad \text{and analogously for } M_k^{NW}, 
\end{align*}
and let $M_k = M_k^{NE}\cap M_k^{NW}.$ Define also
\[
 \tD_k: |Y_{*,k}| \leq \frac{C}{\tr_{q_k}} \text{ for } *=E,N,W,S.
\]
\[
  E_{1,k}: \sum_{m>c_{\tgt}q_k/(1+\zeta)} R_{k,m}^*\eta^m \leq c_1(\eta^3\theta^{-2})^{c_\theta \delta k} (\mu^{-1}\eta)^k, \text{ for } *=E,N,W,S.
\]
with $c_1$ to be specified and $\delta$ from \eqref{params}, noting that by the bound on $\eta$ in \eqref{params} we have $\eta^3\theta^{-2}<1$, and
\begin{align*}
  E_{2,k}: &\text{ there are no slow bonds in any $J_{*,k}$ ($*=$ N, E, S, W)}\\
    &\qquad\text{nor in any $H_{*,\cdot,k}$ ($*=$ NE, NW, $\cdot=$ U, L)}.
\end{align*}
Finally, we set $F_k = I_k \cap M_k \cap \tD_k \cap E_{1,k} \cap E_{2,k}$. Note that when $F_k$ occurs we have $\gO_k\subset\Theta_k$.

\begin{lemma}
\label{lem:success}
There exists some $\kappa_1$ such that
\begin{equation}\label{Fkdens}
  \liminf_{n\to\infty}   \frac 1n \sum_{k=1}^n 1_{F_k} \geq\kappa_1 \quad {\rm a.s.}
  \end{equation}
 Moreover, letting $n_1(\go)<n_2(\go) <\cdots$ be the indices for which $\go\in  F_k$, we have
 \begin{equation}\label{nj}
 \limsup_{j\to\infty} \frac{n_{j+1}}{n_{j}} =1 \quad {\rm a.s.}
\end{equation}  
\end{lemma}

\begin{proof}
In the events $I_k$ and $M_k$, the highways $H_{*,\cdot,k}$ and associated values $X_{i,k}^*$ are taken from the configuration $\omega^{\rm zig,thin,2}$.  Using instead the configuration $\omega^{\rm zig,thin,1}$ yields different events, which we denote $I_k^1$ and $M_k^1$ respectively.  Let $F_k^1 = I_k^1 \cap M_k^1 \cap \tD_k$ (noting we do not intersect with $E_{1,k}, E_{2,k}$ here).
First, we prove the following.
\begin{clm}
\label{claim1}
There exists $\kappa_1>0$ such that 
\[
  \liminf_{n\to\infty}   \frac 1n \sum_{k=1}^n 1_{F_k^1} \geq\kappa_1 \quad {\rm a.s.}
  \]
\end{clm}
\begin{proof}[Proof of Claim \ref{claim1}]
As a first step, we show that
\begin{equation}\label{PIklower}
  \inf_{k\geq k_0} \bP(I_k^1\cap M_k^1)>0.
\end{equation}
For a lower SW/NE zigzag highway $H$, let $(x_0(H),0)$ be the intersection point in $A_1$ closest to $(0,0)$ when one exists, and similarly for an upper SW/NE zigzag highway $H$ let $(0,y_0(H))$ be the intersection point in $A_2$ closest to $(0,0)$. We call $H$ \emph{k-connecting} if $H$ intersects both $G_k^1$ and $G_k^3$. The event $I_k^1\cap M_k^1$ contains the event
\[
A_k: \quad
\begin{minipage}{14.3cm}
there exists exactly one lower SW/NE highway $H$ in $\omega$ of class $k$ or more intersecting $\Theta_k^{NE}$, and this $H$ is $k$-connecting and satsifies $\tfrac{c}{r_k}\leq x_0(H)\leq \tfrac{C}{r_k}-23$; further, the analogous statement holds for upper SW/NE highways with $y_0(\cdot)$ in place of $x_0(\cdot)$, and for lower and upper SE/NW highways. 
\end{minipage}
\]
The parts of $A_k$ for the 4 types of zigzag highways (upper vs lower, SW/NE vs SE/NW) are independent, so to bound the probability of $A_k$ we can consider just one of these parts and take the 4th power of the corresponding probability.  In particular, 
considering lower SW/NE highways $H$ of class $\ell\geq k$ intersecting $\Theta_k^{NE}$, there are at most $2^{\ell+3}$ possible lengths for $H$, at most $2^{\ell+3}C/r_k$ possible SW-most points, and the choice of $H$-start or $V$-start, so at most $2^{2\ell+7}C/r_k$ possible highways.  Also, considering lower SW/NE highways $H$ of class $m\geq k$ which are $k$-connecting and satisfy $\tfrac{c}{r_k}\leq x_0(H)\leq \tfrac{C}{r_k}-23$, there are at least $2^{m+1}(C-c)/r_k$ possible lower endpoints, and $2^{m+1}$ possible lengths, for $H$.
Therefore considering only lower highways, and since $c<C/2$ we have
\begin{align}\label{IkMk}
  \bP(I_k^1\cap M_k^1)^{1/4} \geq \bP(A_k)^{1/4} &\geq \sum_{m\geq k} \frac{2^{2m+2}(C-c)}{r_k}\ \frac{\theta^m}{2^{2m+4}} 
    \prod_{\ell\geq k} \left( 1 - \frac{\theta^\ell}{2^{2\ell+4}} \right)^{2^{2\ell+7}C/r_k} \notag\\
  &\geq  \left( \sum_{m\geq k} \frac{C\theta^m}{8r_k} \right) \exp\left( -\sum_{\ell\geq k}\frac{16C\theta^\ell}{r_k} \right) \notag\\
  &= \frac{C}{8}e^{-16C},
\end{align}
proving \eqref{PIklower}.

A similar but simpler proof also using a count of highways yields that $\inf_{k\geq k_0} \bP(\tD_k)>0$, so by independence we have
\begin{equation}\label{PFklower}
  \inf_{k\geq k_0} \bP(F_k^1)\geq \kappa_0>0
\end{equation}
for some $\kappa_0$.
But Claim \ref{claim1} is a stronger statement, and we now complete its proof, using \eqref{PFklower}.

\smallskip
Let $\mH_k$ denote the set of all SW/NE highways in $\omega$ which intersect $\Theta_k^{NE}$, together with all SE/NW highways which intersect $\Theta_k^{NW}$. Let $\mF_k$ denote the $\sigma$-field generated by $\mH_k$, and let $\mC_k$ denote the largest $j$ such that $\mH_k$ contains a highway of class $j$.  For $m\geq k$,  class-$m$ SW/NE highways intersecting $\Theta_k^{NE}$ have at most $(2C/r_k)(2^{m+3}+2^{k+3})$ possible SW-most endpoints, and $2^{m+3}$ possible lengths, and 2 directions for the initial step (V- or H-start), so the number of such highways 
is bounded by a sum of Bernoulli random variables, each with parameter (success probability) less than $1/2$ and with total mean at most 
\[
  2^{m+4}\frac{2C}{r_k}(2^{m+3}+2^{k+3})\frac{\theta^m}{2^{2m+4}} \leq \frac{32C\theta^m}{r_k}.
\]
The same is true for SE/NW highways.
Hence for $n\geq 0$ the number of highways of class at least $k+n$ is bounded by a similar sum of Bernoulli variables of total mean at most $32Cr_{k+n}/r_k$.  It follows that for any $n\geq 0$,
\begin{align}\label{Ckbound}
  \bP(\mC_k-k \geq n) &\leq \bP(\text{some highway in $\omega$ of class $k+n$ or more intersects $\Theta_k$})\notag\\
  & \leq \frac{32Cr_{k+n}}{r_k} = 32C\theta^n.
\end{align}

Let $n_0$ be the least integer with $\theta^{n_0}<c/C$.
Define a random sequence of indices $1=K_1<K_2<\dots$ inductively as follows:  
having defined $K_i$, let $K_{i+1}:=\max\{\mC_{K_i} , (K_i+n_0) \} +1$. Here $k\geq K_i+n_0$ ensures $C/r_{K_i} < c/r_k$, and $k>\mC_{K_i}$ ensures that $H_{NE,L,k},H_{NE,U,k}$ do not intersect $\Theta_{K_i}^{NE}$, and likewise for NW in place of NE. For any $j,\ell$, and any $A\in\mF_j$,  the event $A\cap\{K_i=j,\mC_j=\ell\}$ only conditions zigzag highways in $\mH_j$, and ensures that no SW/NE (or SE/NW) highways of class $\ge k = \max(j+n_0,\ell)+1$ intersect $\Theta_j^{NE}$ (or $\Theta_j^{NW}$, respectively); in particular it ensures that $\hX_{i,k}^*>C/r_j$ for $*=$ NE, NW.  Therefore since $C/r_j < c/r_k$ this event increases the probability of $I_k^1\cap M_k^1$: for such $j,\ell,k,A$,
\[
 \bP \big (I_k^1\cap M_k^1 \mid A\cap\{K_i=j,\mC_j=\ell\} \big ) \geq \bP( I_k^1\cap M_k^1).
\]
Similarly the bound in \eqref{Ckbound} is still valid conditionally: for such $j,\ell,k$,
\begin{equation}\label{Ckbound2}
  \bP \big (\mC_k-k \geq n \mid A\cap\{K_i=j,\mC_j=\ell\} \big ) \leq 32C\theta^n \quad\text{for all } n\geq 0.
\end{equation}
It follows that for some $c_2>0$,
\begin{equation}
\label{eq:Ki/i}
  \limsup_{i\to\infty} \frac{K_i}{i} \leq c_2\quad {\rm a.s.}
\end{equation}

On the other hand, it is straightforward that for some $\kappa_2>0$, for all $j<k$ and all $B\in\sigma(\tD_1,\dots,\tD_j)$,
\[
  \bP(\tD_k \mid B ) \geq \kappa_2\bP(\tD_k) \quad {\rm a.s.},
\]
so by independence of HV highways from zigzag ones, using \eqref{PFklower}, 
\begin{equation}\label{Fkcond}
  \bP \big (F_k^1 \mid A\cap B \cap \{K_i=j,\mC_j=\ell\} \big ) \geq \kappa_2\bP(F_k^1) \geq \kappa_2\kappa_0.
\end{equation}
Since $A,B$ are arbitrary, it follows that the variables $1_{F_{K_i}}$ dominate an independent Bernoulli sequence with parameter $\kappa_2\kappa_0$, so
\[
  \liminf_{m\to\infty} \frac 1m \sum_{i=1}^m 1_{F_{K_i}^1} \geq \kappa_2\kappa_0\quad {\rm a.s.}
\]
This, combined with \eqref{eq:Ki/i}, proves Claim \ref{claim1} with $\kappa_1=\kappa_2\kappa_0/c_2$.
\end{proof}
Note that by \eqref{Ckbound2}, the variables $K_{i+1}-K_i, i\geq 1$ are dominated by an i.i.d.~sequence of the form ``constant plus geometric random variable.''

\smallskip
Let us go back to the proof of Lemma \ref{lem:success}.  Let $B_k$ denote the event that none of the 4 highways $H_{*,\cdot,k}(\omega^{\rm zig,thin,1})$ are deleted in stage-2 deletions, then we have
\begin{equation}\label{nodelete}
  I_k^1 \cap M_k^1 \cap B_k \subset I_k\cap M_k.
\end{equation}

\begin{clm}\label{claimio}
We have
\[
  \sum_k \bP(B_k^c\cap I_k^1) < \infty\ ; \quad \sum_k \bP(E_{2,k}^c\cap \tD_k) < \infty \quad \text{ and } \quad   \sum_k \bP(E_{1,k}^c) < \infty.
\]
\end{clm}

\begin{proof}[Proof of Claim \ref{claimio}]
Let $Z_{*,\cdot,k}^H$ denote the class of the zigzag highway $H_{*,\cdot,k}(\omega^{\rm zig,thin,1})$, for $*=$ NE, NW and $\cdot=$ U, L, and let $Z_{*,k}^J$ denote the class of the HV highway $J_{*,k}$ for $*=$ N, E, W, S. 
Then, by \eqref{stage2prob}, for $n\ge 0$
\begin{equation}\label{eq}
  \bP \big (B_k^c \mid I_k^1\cap \{Z_{NE,L,k}^H = k+n\} \big ) \leq \frac{8}{1-\tgt}\ 2^{-\zeta(k+n)} \le \frac{8}{1-\tgt}\ 2^{-\zeta k}\, .
\end{equation}
Since the upper bound is independent of $n$, summing over $n$ we get that $ \bP(B_k^c \cap I_k^1)  \le 8 (1-\tgt)^{-1} 2^{-\zeta k}$, and the first item of Claim \ref{claimio} is proven.

Similarly,
\[
  \bP \big (E_{2,k}^c \mid \tD_k\cap \{Z_{N,k}^J = q_k+n\} \big ) \leq 2^{q_k+n+3}4^{-(q_k+n)} \leq 8\cdot 2^{-q_k}.
\]
so we get that $\bP\big( E_{2,k}^c \cap \tD_k \big) \le 2^{- q_k +3}$, and the second item of Claim \ref{claimio} is proven.

\smallskip
For the last item, recall $c_{\tgt}/(1+\zeta) = c_{\tgt}-\delta$. We have from the upper bound for $\mu$ in \eqref{params2}
\begin{align}\label{Rkmsum}
  \sum_{m>c_{\tgt}q_k/(1+\zeta)} \theta^{m-k-c_\theta \delta m}\eta^m 
    &\leq c_1 (\theta^{1-c_\theta \delta}\eta)^{(c_\theta (c_{\tgt}-\delta)-1)k}\mu^k \theta^{-c_\theta \delta k} (\mu^{-1}\eta)^k \notag\\
  &\leq c_1 (\eta^3 \theta^{-2})^{c_\theta \delta k} (\mu^{-1}\eta)^k,
\end{align}
so 
\[  \bP(E_{1,k}^c) \leq 4\sum_{m>c_{\tgt}q_k/(1+\zeta)} \bP\left( R_{k,m}^N > \theta^{m-k-c_\theta \delta m} \right) \, .\]
Analogously to \eqref{Ckbound} we have that
$\bE(R_{k,m}^N) \leq c_3\theta^m /r_k$
and by Markov's inequality,
\[
  \bP(E_{1,k}^c)
    \leq c_4 \sum_{m>c_{\tgt}q_k/(1+\zeta)} \theta^{c_\theta \delta m} \leq c_5\theta^{c_6k}.
\]
This proves the last item of Claim \ref{claimio}.
\end{proof}

Equation \eqref{Fkdens} in Lemma \ref{lem:success} now follows from \eqref{nodelete}, Claims \ref{claim1} and \ref{claimio}, and the Borel-Cantelli lemma. Equation \eqref{nj} comes additionally from the remark made at the end of the proof of Claim~\ref{claim1}.
\end{proof}

\subsection{Properties of geodesics in case of a success}
\label{sec:properties}

For $x\notin \gO_k$, we let $ \Gamma_{0x}$ be the geodesic from $0$ to $x$ (unique since the $\xi_e$ are continuous random variables.) For $p,q\in  \Gamma_{0x}$ we denote by $ \Gamma_{0x}[p,q]$ the segment of $\gG_{0x}$ from $p$ to $q$. We let $p_x$ be the first point of $ \Gamma_{0x}$ outside $\gO_k$.  We then define $t_x$ in the boundary of the ``near-rectangle'' $\gO_k^{NE}\cap\gO_k^{NW}$ as below.  

Note that this boundary consists of 4 zigzag segments, one from each highway $H_{*,\cdot,k}$ ($*=$ NE, NW; $\cdot=$ U, L).  Some boundary points (one bond or site at each ``corner'') are contained in 2 such segments; we call these \emph{double points}.  Removing all double points leaves 4 connected components of the boundary, which we call \emph{disjoint sides} of $\gO_k^{NE}\cap\gO_k^{NW}$, each contained in a unique highway $H_{*,\cdot,k}$. The set $\gO_k \backslash (\gO_k^{NE}\cap\gO_k^{NW})$ has 4 connected components, which we call \emph{arms}, extending from $\gO_k^{NE}\cap\gO_k^{NW}$ in the directions NW, NE, SE, SW.  Each arm includes one disjoint side of $\gO_k^{NE}\cap\gO_k^{NW}$  (recalling that the sets $\gO_k^*$ are open.) 

 If $p_x$ is not a double point then it is contained in the boundary of one of the arms, and we let $t_x$ be the first point of $\gG_{0x}[0,p_x]$ in that arm (necessarily in a disjoint side; see Case 3 of Figure \ref{fig:3cases}.) If instead $p_x$ is a double point, then we pick arbitrarily one of the two highways $H_{*,\cdot,k}$ containing it, and let $t_x$ be the first site of $\gG_{0x}[0,p_x]$ in that highway.

For $\mu$ from \eqref{params2}, let $L_{N,k}$ denote the horizontal line $\RR\times\{\mu^{-k}\}$  which, at least on the event $F_k$, lies above $J_{N,k}$ and below $G_k^1$ (since $1/2 < \mu < \tilde \theta^{c_{\theta}}$); $L_{*,k}$ is defined analogously for $*=$ E, S, W. 
We define $\Lambda_{H,k}$ to be the ``horizontal axis corridor,'' meaning the closure of the portion of $\gO_k^{NE}\cap\gO_k^{NW}$ strictly between $L_{S,k}$ and $L_{N,k}$, and let $\Lambda_{V,k}$ denote the similar ``vertical axis corridor.''  The primary part of establishing directedness in axis directions is showing that all semi-infinite geodesics remain in these corridors until they exit out the far end, at least for many $k$, via the following deterministic result.

\begin{lemma}
\label{lem:geodesics}
For sufficiently large $k$, when $ F_k$ occurs, for all $x\notin  \Omega_k$ we have either $\gG_{0x}[0, t_x]\subset {\Lambda}_{H,k}$ or $\gG_{0x}[0, t_x]\subset {\Lambda}_{V,k}$. 
\end{lemma}

\begin{proof}
Let us start with a claim analogous to what we proved in Section \ref{sec:simple}. Recall the definitions of $\Gamma_{0x}$ and~$p_x$.
\begin{clm}
\label{claim3}
When $F_k$ occurs, for all $x\notin \gO_k$ we have $p_x \notin G_k^1\cup G_k^2\cup G_k^3\cup G_k^4$.
\end{clm}
\begin{proof}[Proof of Claim \ref{claim3}]
Suppose $p_x=(r,s)$ is in the horizontal part of $G_k^1$ (so $r\geq s=2^k$) and let ${U}_k'$ be the upper endpoint of the bond which is the intersection of $H_{NE,L,k}$ and the vertical line through $p_x$.  Define
the alternate path $\pi_x$ from 0 east to $U_k^{NE}$, then NE along $H_{NE,L,k}$ to $U_k'$, then north to $p_x$.  We now compare the passage times of $\gG_{0x}[0,p_x]$ versus $\pi_x$.

\smallskip
We divide the bonds of the lattice into NW/SE \emph{diagonal rows}: the $j$th diagonal row $R_j$ consists of those bonds with one endpoint in $\{(x_1,x_2):x_1+x_2=j-1\}$ and the other in $\{(x_1,x_2):x_1+x_2=j\}$.  
We call a bond $e\in\gG_{0x}[0,p_x]$ a \emph{first bond} if for some $j\geq 1$, $e$ is the first bond of $\gG_{0x}[0,p_x]$ in $R_j$, and we let $N$ be the number of first bonds in $\gG_{0x}[0,p_x]$ which are in SE/NW highways.  There are at least $2s$ first bonds, and any first bond $e$ not in any SE/NW highway satisfies $\tau_e \geq 0.7+0.1\eta^{k-1}$, since from $F_k\subset M_k$ we have $k(e)\leq k-1$.  From \eqref{six}, letting $q=| \gG_{0x}[0,p_x]|\geq 2s$, at most $q/24$ SE/NW highways intersect $\gG_{0x}[0,p_x]$. Since no two first bonds can be in the same SE/NW highway,  we thus have $N\leq q/24$.  Therefore using \eqref{alphasum},
\begin{align}\label{slower}
  T( \gG_{0x}[0,p_x]) &\geq (2s-N)(0.7+0.1\eta^{k-1}) + (q-2s+N)0.7 - 0.2 \notag\\
  &= 2s(0.7+0.1\eta^k) + 2s(0.1\eta^{k-1}-0.1\eta^k) + (q-2s)0.7 - 0.1N\eta^{k-1} - 0.2 \notag\\
  &\geq 2s(0.7+0.1\eta^k) + 2s(0.1\eta^{k-1}-0.1\eta^k) + (q-2s)0.7 - \frac{q}{240}\eta^{k-1} - 0.2.
\end{align}
If $q\geq 3s$ then $q(0.7-\eta^{k-1}/240) \geq 1.4s$ so from \eqref{slower} we get 
\begin{equation}\label{slower2}
  T( \gG_{0x}[0,p_x]) \geq 2s(0.7+0.1\eta^k) + 2s(0.1\eta^{k-1}-0.1\eta^k) - 0.2.
\end{equation}
If instead $q<3s$ then since $1-\eta\geq 1/8$ we have $s(0.1\eta^{k-1}-0.1\eta^k) \geq s\eta^{k-1}/80 \geq q\eta^{k-1}/240$, so
\begin{equation}\label{slower3}
  T( \gG_{0x}[0,p_x]) \geq 2s(0.7+0.1\eta^k) + s(0.1\eta^{k-1}-0.1\eta^k) -  0.2.
\end{equation}
By contrast, the NE segment of $\pi_x$ has length $2(r-X_{1,k}^{NE})$, the N segment has length $X_{1,k}^{NE}-(r-s)$, and all bonds have passage times at most 1.4, so we have for large $k$
\begin{align*}
  T(\pi_x) &\leq 1.4 X_{1,k}^{NE} + 1.4(X_{1,k}^{NE} - (r-s)) + 2(0.7+0.1\eta^k + 0.1\tgh^k)(r-X_{1,k}^{NE}) \\
  &\le 2.8  X_{1,k}^{NE} + 2s(0.7+0.1\eta^k +0.1 \tgh^k) \\
  &\leq 2s(0.7+0.1\eta^k) + \frac{4C}{r_k},
\end{align*} 
where we used that $X_{1,k}^{NE} \le C/r_k$, and that $ \tgh^k s = (2\tgh)^k \le C/r_k$, using \eqref{params}.
Therefore by \eqref{etatheta} we get
\[
  T(\gG_{0x}[0,p_x]) - T(\pi_x) \geq 0.1\cdot 2^k(\eta^{k-1}-\eta^k) - \frac{4C}{r_k} - 0.2 > 0.
\] 
This contradicts $\gG_{0x}[0,p_x]$ being a geodesic, so $p_x$ cannot be in the horizontal part of $G_k^1$. All other cases are symmetric, so Claim \ref{claim3} is proved.
\end{proof}

We need to further restrict the location of $ \gG_{0x}[0,p_x]$.  When $F_k$ occurs, we begin by dividing the sites of each $H_{*,\cdot,k}$ (with $*=$ NW, NE and $\cdot=$ L, U) into \emph{accessible} and \emph{inaccessible} sites
In $H_{NE,L,k}$ we define as \emph{inaccessible} the sites strictly between its intersection with $J_{N,k}$ and its intersection with $J_{W,k}$; the rest of the sites are \emph{accessible}.  Lattice symmetry yields the definition of accessible in the other 3 zigzag paths.

This definition enables us to define canonical paths to reach accessible points, which we need below.
Given a site $a\in J_{N,k}\cap\gO_k^{NE}$ and an accessible site $b\in H_{NE,L,k}$ in the first quadrant, there is a \emph{canonical path} from $a$ to $b$ which follows $J_{N,k}$ from $a$ to $H_{NE,L,k}$, then (changing direction 45 degrees) follows $H_{NE,L,k}$ to $b$.  Similarly we can define canonical paths from sites $a\in J_{S,k}\cap\gO_k^{NW}$ to accessible $b\in H_{NW,U,k}$ in the 4th quadrant, with further extension by lattice symmetries.
When two points $a,b$ lie in the same horizontal or vertical line, we define the \emph{canonical path} from $a$ to $b$ to be the one which follows that line.

We say $\gG_{0x}[0,p_x]$ is \emph{returning} if it contains a point of $L_{*,k}$, followed by a point of $J_{*,k}$, where both values $*$  (N, E, W or S) are the same. We recall that by definition, from \eqref{params2} we have $\theta<\mu<\tgt^{c_\theta }$ which ensures that, when $F_k$ occurs, the height of $\Lambda_{H,k}$ is much less than its length, but much more than the height of $J_{N,k}$.
Recall the definitions of $p_x,t_x$ from the beginning of the section.


\begin{clm}\label{claimcases}
If $F_k$ occurs and $\Gamma_{0x}[0,t_x]\not\subset {\Lambda}_{H,k} \cup {\Lambda}_{V,k}$, then either $\gG_{0x}[0,p_x]$ is returning, or at least one of $t_x,p_x$ is accessible.
\end{clm}
\begin{proof}
Suppose $F_k$ occurs, $ \Gamma_{0x}[0,t_x]\not\subset {\Lambda}_{H,k} \cup {\Lambda}_{V,k}$, and neither $t_x$ nor $p_x$ is accessible. We may assume $p_x$ lies in $H_{NE,L,k}$ on or above $H_{NW,U,k}$, as other cases are symmetric; then $t_x\in H_{NW,U,k}$ (or we may assume so, if $p_x$ is a double point.)  Since $ \Gamma_{0x}[0,t_x]\not\subset {\Lambda}_{H,k} \cup {\Lambda}_{V,k}$, $ \Gamma_{0x}[0,t_x]$ must intersect $L_{N,k}\cup L_{S,k}$; 
let $q_x$ be the first such point of intersection.  
If $q_x\in L_{S,k}$, then the fact that $t_x$ is not accessible (it must lie above $J_{S,k}$) means  that $\gG_{0x}[0,p_x]$ is returning. If $q_x\in L_{N,k}$, then the fact that $p_x$ is not accessible (hence lying below $J_{N,k}$)  again means that $\gG_{0x}[0,p_x]$ is returning. This proves Claim \ref{claimcases}.
\end{proof}

We now complete the proof of Lemma \ref{lem:geodesics}, by a contradiction argument. Assume that $F_k$ occurs, but for some $x\notin \gO_k$ we have $ \Gamma_{0x}[0,t_x] \not\subset {\Lambda}_{H,k} \cup {\Lambda}_{V,k}$.  As in the proof of Claim~\ref{claimcases}, we may assume $p_x$ lies in $H_{NE,L,k}$ on or above $H_{NW,U,k}$, and then that $t_x\in H_{NW,U,k}$.  Let $q_x\in L_{N,k}\cup L_{S,k}$ be as in the proof of Claim~\ref{claimcases}, and let $a_x$ be the last point of $ \Gamma_{0x}[0,q_x]$ in~$J_{*,k}$, with subscript $*=$ N or S according as $q_x\in L_{N,k}$ or $q_x\in L_{S,k}$. Thus $ \gG_{0x}[0,p_x]$ follows a path $0\to a_x\to q_x\to t_x\to p_x$. 

By Claim \ref{claimcases}, we now have three cases:

\begin{description}
\item[{Case 1.}] $ \gG_{0x}[0,p_x]$ is returning.  In this case we define $b_x$ to be the first point of $ \Gamma_{0x}[q_x,p_x]$ in the line $J_{*,k}$ ($*=$ N or S) containing $a_x$.
\item[{Case 2.}] $ \gG_{0x}[0,p_x]$ is not returning, and $p_x$ is accessible (hence above $J_{N,k}$, so necessarily $ q_x \in L_{N,k}$).  In this case we define $b_x=p_x$.  
\item[{Case 3.} ]$ \gG_{0x}[0,p_x]$ is not returning, $p_x$ is inaccessible (hence below $J_{N,k}$, so necessarily $ q_x \in L_{S,k}$), and $t_x$ is accessible.  
In this case we define $b_x=t_x$.
\end{description}

In all three cases we compare passage times for $ \gG_{0x}[a_x,b_x]$ to that of the canonical path from $a_x$ to $b_x$, which we denote $\gamma_x$, and we obtain our contradiction by showing that 
\begin{equation}\label{faster}
  T(\gamma_x) < T( \gG_{0x}[a_x,b_x]),
\end{equation}
see Figure \ref{fig:3cases}.

\begin{figure}[htbp]
\vspace{-.1in}
\begin{center}
\includegraphics[scale=.9]{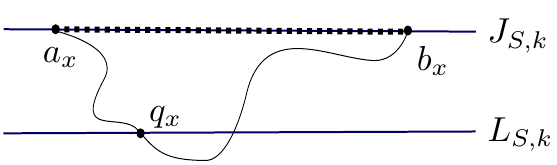}\quad
\includegraphics[scale=.9]{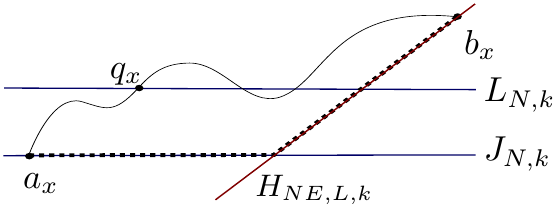}\quad
\includegraphics[scale=.85]{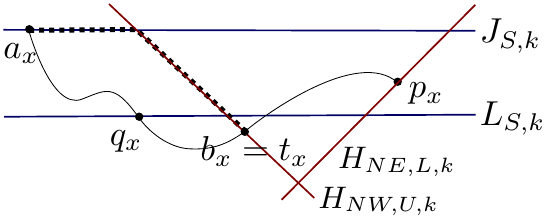}
\end{center}
\caption{\footnotesize 
Cases 1-2-3 (from left to right). The path $ \Gamma_{0x}$ is the curved line going through $ a_x, q_x$ and $ b_x$, and the canonical path $\gamma_x$ is the dashed line joining $ a_x$ to $ b_x$ through highways.}
\label{fig:3cases}
\end{figure}

Since $E_{2,k}$ occurs, there are no slow bonds in $\gamma_x$, so to prove \eqref{faster} we may and do assume there are no slow bonds at all.
Cases~2 and~3 are essentially symmetric, so we focus on  Case~2 first, then the simpler Case~1.

We write (in Case 2)
\begin{equation}\label{coords}
 a_x = (u,Y_{N,k}), \quad q_x = (y,\mu^{-k}), \quad 
   b_x = p_x  = (r,s),
\end{equation}
and we let $(d,Y_{N,k})$ denote the left endpoint of the bond $J_{N,k}\cap H_{NE,L,k}$.

The class of the highway $J_{N,k}$ is at least $q_k$, so due to stage-2 deletions, the only zigzag highways intersecting $J_{N,k}$ have class greater than $c_{\tgt}q_k/(1+\zeta)$.  
From \eqref{alphasum}, \eqref{tenthbound} and the definition of $E_{1,k}$ we have
\begin{align}\label{gammabound3}
  T(\gamma_x) &\leq 0.9(d-u) + 0.4 + \sum_{m>c_{\tgt}q_k/(1+\zeta)} R_{k,m}\eta^m + 2(s-Y_{N,k})(0.7+0.1\eta^k) + 0.2 \notag\\
  &\leq 0.9(d-u) + 2(s-Y_{N,k})(0.7+0.1\eta^k) + 0.6 + c_1(\eta^3\theta^{-2})^{c_\theta \delta k} (\mu^{-1}\eta)^k.
\end{align}
Our aim is to show that for some $c_7$,
\begin{equation}\label{Gammaslow}
  T( \gG_{0x}[a_x,b_x]) \geq 0.9(d-u) + 2(s-Y_{N,k})(0.7+0.1\eta^k)  + c_7 (\mu^{-1}\eta)^k,
  \end{equation}
which with \eqref{gammabound3} is sufficient to yield \eqref{faster} for all large $k$, since $\mu^{-1}\eta>1$ by \eqref{params2} and $\eta^3\theta^{-2}<1$ by \eqref{params}.

\smallskip
The rest of the proof is devoted to showing \eqref{Gammaslow}, by analyzing the type (and number of each type) of bonds that the path $\Gamma_{0x}$ uses.
 
We observe first that the intersection of a geodesic with any one zigzag highway is always connected, since we have assumed there are no slow bonds.
A \emph{singleton bond} in a path $\Gamma$ is a zigzag bond in $\Gamma$ which is preceded and followed in $\Gamma$ by boundary bonds.  
We divide $ \gG_{0x}[a_x,b_x] = (a_x=z_0,z_1,\dots,z_n=b_x)$ into the following types of segments: 
\begin{itemize}
\item[(i)] \emph{zigzag segments}: maximal subsegments which do not contain two consecutive non-zigzag bonds.  A zigzag segment must start and end with a boundary bond, unless it starts at $a_x$ or ends at $b_x$.
\item[(ii)] \emph{intermediate segments}: the segments in between two consecutive zigzag segments, the segment up to the first zigzag segment, and the segment after the final zigzag segment (any of which may be empty.)
\end{itemize}
Within zigzag segments we find
\begin{itemize}
\item[(iii)] \emph{component segments}: maximal subsegments contained in a single zigzag highway.  
\end{itemize} 
For each component segment in $ \gG_{0x}[a_x,b_x]$ that is not both an intersection bond and a singleton bond, there is 
a unique highway orientation (SW/NE or SE/NW) determined by the zigzag highway containing the segment.  We define $\Phi(e)$ to be this orientation, for each bond $e$ in the segment.  For singleton bonds in $\Gamma$ that are also intersection bonds, we assign $\Phi(e)$ arbitrarily. We say that a zigzag highway $H$ intersects $ \gG_{0x}[a_x,b_x]$ \emph{redundantly} if the intersection is a single bond $e$ and the orientation of $H$ is not $\Phi(e)$.

\smallskip
We let $z_{n_{2j-2}},z_{n_{2j-1}}$ be the endpoints of the $j$th intermediate segment, $1\leq j\leq J+1$, so $\beta_j=  \gG_{0x}[z_{n_{2j-1}},z_{n_{2j}}]$ is the $j$th zigzag segment, $1\leq j \leq J$.   
For any path $\Gamma$ we define
\begin{equation}
  T_{\alpha^*}(\Gamma) = \sum_{e\in\Gamma} \alpha_e^*. \quad 
  \end{equation}
For $*=$ Z, B, H, V we define 
\[
  N_*(\Gamma) = \big| \{e\in\Gamma: e \text{ has property } * \} \big|,
\] 
where subscripts and corresponding properties are as follows:
\begin{align*}
  Z: \text{ zigzag}& ,\qquad
  B: \text{ not zigzag},\qquad
  H: \text{ horizontal}, \qquad
  V: \text{ vertical},\\
  &N: \text{ northward step in $\Gamma$ (E, W, S similar)}.
\end{align*}
We may combine subscripts to require multiple properties, for example $N_{ZH}(\Gamma) = | \{e\in\Gamma: e $ is a horizontal zigzag bond$\}|$, and we use superscripts NE or NW to restrict the count to zigzag bonds with $\Phi(e) =$ SW/NE or SE/NW, respectively. We also let $N_{Hi}(\Gamma)$ be the number of zigzag highways intersecting $\Gamma$ non-redundantly.
Note that if for example a SW/NE highway segment is traversed by $\Gamma$ in the NE direction, the number of N and E steps differs by at most 1.  
It follows that for every geodesic $\Gamma$ we have
\begin{align}\label{NHi}
  N_{Hi}(\Gamma) \geq D(\Gamma) := |N_{ZE}^{NE}&(\Gamma) - N_{ZN}^{NE}(\Gamma)| 
    + |N_{ZW}^{NE}(\Gamma) - N_{ZS}^{NE}(\Gamma)| \notag\\
  &+ |N_{ZE}^{NW}(\Gamma) - N_{ZS}^{NW}(\Gamma)| + |N_{ZW}^{NW}(\Gamma) - N_{ZN}^{NW}(\Gamma)|.
\end{align}


We now make some observations about geodesics and zigzag segments. 
We show that a zigzag segment~$\beta_j$ may contain multiple bonds of at most one zigzag highway (which we call \emph{primary}, when it exists)---any zigzag bonds in $\beta_j$ not in the primary highway are necessarily singleton bonds, and there are at most 2 of these; if $\beta_j$ intersects 3 zigzag highways non-redundantly then the primary highway must lie between the two singletons. Moreover, any non-zigzag interior bond of some $\beta_j$ must be an entry/exit bond.  
Indeed, due to the stage-3 trimming,  in order for $\Gamma$ to switch from one zigzag highway to another within $\beta_j$ (with at least two bonds on each) there would need to be one of the following succession of steps (or some lattice rotation thereof): (i) N, E, E, E, S with the middle one being an exit/entry bond; (ii) N, E, E, S with the middle 2 being meeting zigzag bonds; (iii) N, E, S with the middle one being an intersection bond. We refer to Figure \ref{fig:crossing} for a picture. But  (since we are assuming no slow bonds) none of these patterns can occur in a geodesic, because omitting the N and S steps always produces a faster path.
Then \eqref{six} and the stage-3 trimming establish the remaining properties mentioned.

For $2\leq j\leq J-1$, if $\beta_j$ contains a primary highway then the bonds of $\beta_j$, from the initial bond through the first bond of the primary highway, must follow one of the following patterns (we refer to figure~\ref{fig:crossing}):
\begin{itemize}
\item[(i)] skimming boundary, terminal
\item[(ii)] boundary (not skimming), zigzag (not intersection)
\item[(iii)] semislow boundary, intersection, meeting
\item[(iv)] boundary (not skimming), meeting, meeting
\item[(v)]   boundary (not skimming), meeting, intersection (followed by meeting)
\item[(vi)] boundary (not skimming), zigzag (not intersection), entry/exit, zigzag (not intersection).
\end{itemize}
The same is true in reverse order at the opposite end of $\beta_j$. (Note that certain patterns cannot appear in a geodesic $\Gamma$, for example a semislow boundary bond with both endpoints in meeting bonds cannot be adjacent in $\Gamma$ to either meeting bond, so these are not listed here.) 
If there is no primary highway then the full $\beta_j$ follows one of the following patterns, or its reverse:
\begin{itemize}
\item[(vii)] boundary (not skimming), zigzag (not intersection), boundary (not skimming)
\item[(viii)] skimming boundary, singly terminal, boundary (not skimming)
 \item[(ix)] skimming boundary, doubly terminal, skimming boundary 
\item[(x)] semislow boundary, intersection, semislow boundary
\item[(xi)] normal boundary, zigzag (not intersection), entry/exit, zigzag (not intersection), normal boundary.
\end{itemize} 
It is readily checked from this that in all cases
\begin{align}\label{middlej}
  T_{\alpha^*}(\beta_j) &\geq 0.9N_B(\beta_j) + 0.7N_Z(\beta_j) + 0.2N_{Hi}(\beta_j), \quad 2\leq j\leq J-1.
\end{align}
For $j=1,J$, $\beta_j$ is a truncation of a path as in (i)--(xi), omitting a (possibly empty) segment of bonds at one end,
and we similarly have
\[
  T_{\alpha^*}(\beta_j) \geq 0.9N_B(\beta_j) + 0.7N_Z(\beta_j) + 0.2N_{Hi}(\beta_j) - 0.2, \quad j=1,J,
\]
and therefore
\begin{align}\label{TGlower}
  &T_{\alpha^*}( \gG_{0x}[a_x,b_x])  \geq 0.9N_B( \gG_{0x}[a_x,b_x]) 
     + 0.7N_Z( \gG_{0x}[a_x,b_x]) + 0.2N_{Hi}( \gG_{0x}[a_x,b_x]) - 0.4 .
\end{align}
By \eqref{TGlower} and \eqref{NHi} we have
\begin{align}\label{TGlower0}
   T_{\alpha^*}(\gG_{0x}[a_x,b_x]) \geq 0.9N_B(\gG_{0x}[a_x,b_x]) 
     + 0.7N_Z(\gG_{0x}[a_x,b_x]) + 0.2D(\gG_{0x}[a_x,b_x]) - 0.4,   
\end{align}
and from the definition of $M_k^{NE}$,
\begin{align}\label{TGeta}
   T(\gG_{0x}[a_x,b_x]) &\geq T_{\alpha^*}(\gG_{0x}[a_x,b_x]) 
     + 0.1\eta^{k-1} \big(   N_{ZN}^{NE}(\gG_{0x}[a_x,b_x]) + N_{ZE}^{NE}(\gG_{0x}[a_x,b_x])   \big).
\end{align}

In view of \eqref{coords} and \eqref{TGlower0}-\eqref{TGeta}, let us consider the question of minimizing
\begin{align}\label{obj}
  0.9&(n_{BE} + n_{BW} + n_{BN} + n_{BS}) \notag\\
  &\quad + 0.7\left(n_{ZN}^{NE} + n_{ZE}^{NE} + n_{ZW}^{NE} + n_{ZS}^{NE} 
    + n_{ZN}^{NW} + n_{ZE}^{NW} + n_{ZW}^{NW} + n_{ZS}^{NW} \right) \notag\\
  &\quad + 0.2\Big( |n_{ZE}^{NE} - n_{ZN}^{NE}| 
    + |n_{ZW}^{NE} - n_{ZS}^{NE}| + |n_{ZE}^{NW} - n_{ZS}^{NW}| + |n_{ZW}^{NW} - n_{ZN}^{NW}| \Big) \notag\\
  &\quad + 0.1\eta^{k-1}(n_{ZN}^{NE} + n_{ZE}^{NE}) - 0.4
\end{align}
subject to all 8 variables being nonnegative integers satisfying
\begin{align}
   \label{constraint1}
  n_{ZE}^{NE} + n_{ZE}^{NW} - n_{ZW}^{NE} - n_{ZW}^{NW} + n_{BE} - n_{BW} &= (d-u)+(s-Y_{N,k}) = N_H(\gamma_x), \\
  \label{constraint2}
  n_{ZN}^{NE} - n_{ZS}^{NW} - n_{ZS}^{NE} + n_{ZN}^{NW} + n_{BN} - n_{BS} &= s-Y_{N,k} = N_V(\gamma_x), \\
  \label{constraint3}
  n_{ZN}^{NE} + n_{ZN}^{NW} + n_{BN} &= (s-Y_{N,k}) + g, \\
  \label{constraint4}
  n_{ZS}^{NE} + n_{ZS}^{NW} + n_{BS} &= g, \\
  \label{constraint5}
  n_{ZN}^{NW} + n_{BN} &= j,
\end{align}
for some fixed $g\geq 0$ and $0\leq j\leq (s-Y_{N,k})+g$. Here \eqref{constraint4} is redundant but we include it for ready reference, and despite \eqref{constraint5} we formulate the problem with $n_{ZN}^{NW}$ as a variable, to match the rest of the problem.  Further, $g$ may be viewed as an ``overshoot'', the number of northward steps beyond the minimum needed to reach the height $s$ of $p_x$, and $j$ is the number of northward steps taken ``inefficiently,'' that is, not in NE/SW zigzag highways.  Since $q_x$ is at height $\mu^{-k}$, we may restrict to $s+g\geq \mu^{-k}$, and thus from the definition of $\tilde D_k$, also to
\begin{equation}\label{qxheight}
  s-Y_{N,k}+g\geq \mu^{-k} - \frac{C}{\tilde r_{q_k} } \geq \frac{ \mu^{-k} }{2},
\end{equation} 
the last inequality being valid for large $k$, following from the fact that $\tilde r_{q_k}$ is a constant multiple of $\tgt^{ -c_{\gt}k }$ while $\mu < \tgt^{ c_{\gt} }$ by \eqref{params2}. 
To study this we use the concept of \emph{shifting mass} from one variable $n_{\bullet}^*$ to a second one, by which we mean incrementing the second by 1 and the first by $-1$.  We also use \emph{canceling mass} between two variables in \eqref{constraint1} or two in \eqref{constraint2}, one appearing with $``+"$ and the other with $``-"$, by which we mean decreasing each variable by 1.  

Shifting mass from $n_{BS}$ to $n_{ZS}^{NE}$, or from $n_{BN}$ to $n_{ZN}^{NW}$, does not increase \eqref{obj}, so a minimum exists with $n_{BS}=n_{BN}=0$, so we may eliminate those two variables. 
 Among the variables $n_{\bullet}^*$ in \eqref{constraint1}, if a variable with ``$+$'' and a variable with ``$-$'' are both nonzero, then canceling (unit) mass between them decreases \eqref{obj} by at least 1; this means that a minimum exists with all negative terms on the left in \eqref{constraint1} equal to 0.  Then shifting mass in \eqref{constraint2} from $n_{ZS}^{NE}$ to $n_{ZS}^{NW}$, or in \eqref{constraint1} from $n_{BE}$ to $n_{ZE}^{NW}$, does not increase \eqref{obj} (since the preceding step has set $n_{ZW}^{NE}$ and $n_{ZW}^{NW}$ to $0$), so there is a minimum with also $n_{ZS}^{NE}=n_{BE}=0$.  With these variables set to 0, the problem becomes minimizing 
\begin{align}\label{objrev1}
  0.7&\left(n_{ZE}^{NE} + n_{ZE}^{NW} + s-Y_{N,k}+2g \right) + 0.2\Big( |n_{ZE}^{NE} - (s-Y_{N,k}+g-j)| + |n_{ZE}^{NW} - g| + j\Big) \notag\\
  &\qquad + 0.1\eta^{k-1}(s-Y_{N,k}+g-j + n_{ZE}^{NE}) - 0.4
\end{align}
subject to 
\begin{align}\label{newconstrE}
  n_{ZE}^{NE} + n_{ZE}^{NW} &= (d-u)+(s-Y_{N,k}).
\end{align}
Setting $n_{ZE}^{NE}=z, n_{ZE}^{NW} = (d-u)+(s-Y_{N,k})-z$ and considering the effect of incrementing $z$ by 1, we see that \eqref{objrev1} is minimized (not necessarily uniquely) at 
\begin{equation}\label{zmin}
  z= \min\Big((s-Y_{N,k})+g-j,[(s-Y_{N,k})+(d-u)-g]\vee 0\Big).
\end{equation}
We now consider two cases.

{\bf Case 2A.} $2g-j\leq d-u$.  Here we have $z=(s-Y_{N,k})+g-j\geq 0$ in \eqref{zmin}, and the corresponding minimum value of \eqref{objrev1} is 
\begin{align*}
  0.9&(d-u) + 0.7\cdot 2(s-Y_{N,k}) + g+0.4j + 0.1\eta^{k-1}(2(s-Y_{N,k})+2(g-j)) - 0.4\\
  &\geq 0.9(d-u) + 0.7\cdot 2(s-Y_{N,k}) +g+ 0.1\eta^{k-1} \cdot 2(s-Y_{N,k}) - 0.4\\
  &\geq 0.9(d-u) + (0.7 + 0.1 \eta^{k-1}) 2(s-Y_{N,k})  + 0.1\eta^k(\eta^{-1}-1) \cdot 2 (s-Y_{N,k}+g) - 0.4.
\end{align*}
For the first inequality, we used that $0.4 j\ge  0.2 \eta^{k-1} j$, and for the second one that $g \ge 0.2\eta^k(\eta^{-1}-1)g  $, for $k$ sufficiently large.
With \eqref{TGeta} and \eqref{qxheight} this shows that
\begin{align*}
  T&( \gG_{0x}[a_x,b_x]) \geq 0.9(d-u) + (0.7 + 0.1\eta^k)\cdot 2(s-Y_{N,k}) + 0.1 \eta^k(\eta^{-1}-1)\mu^{-k} - 0.4.
\end{align*}
Since $\eta>\mu$ by \eqref{params2}, this proves \eqref{Gammaslow}.

{\bf Case 2B.} $2g-j>d-u$, so that
\begin{equation}\label{glower}
  0.8g \geq 0.4(d-u) + 0.1\eta^{k-1} j.
\end{equation}
Here we have $z=[(s-Y_{N,k})+(d-u)-g]\vee 0$ in \eqref{zmin}, and the corresponding minimum value of \eqref{objrev1} in the case $z = (s-Y_{N,k})+(d-u) -g>0$ (the case $z=0$ being treated similarly) is 
\begin{align*}
  0.5&(d-u) + 0.7\cdot 2(s-Y_{N,k}) + 1.8g + 0.1\eta^{k-1}((d-u) + 2(s-Y_{N,k}) - j) - 0.4 \\
  &\geq 0.9(d-u) + 0.7\cdot 2(s-Y_{N,k}) + g+  0.1\eta^{k-1}\cdot 2(s-Y_{N,k})  - 0.4
\end{align*}
where the inequality follows from \eqref{glower}.
Then \eqref{Gammaslow} follows as in Case 2A.

\smallskip

We now briefly explain the modifications of the above argument to treat Case~1. Analogously to \eqref{coords}, we write
\[ a_x = (u,Y_{N,k}), \quad q_x = (y,\mu^{-k}), \quad 
   b_x = (v,Y_{N,k}),\]
and we may assume $v-u>0$, since otherwise we can consider the path $\gG_{0x}[a_x,b_x]$ running backwards. 
Similarly to \eqref{gammabound3}, we have
\begin{align}\label{gammabound1}
  T(\gamma_x) &\leq 0.9(v-u) + 0.5 + \sum_{m>c_{\tgt}q_k/(1+\zeta)} R_{k,m}\eta^m \notag\\
  &\leq 0.9(v-u) + 0.5 + c_1(\eta^3\theta^{-2})^{c_\theta \delta k} (\mu^{-1}\eta)^k.
\end{align}
and similarly to \eqref{Gammaslow} we want to show
\begin{equation}\label{Gammaslow2}
  T(\gG_{0x}[a_x,b_x]) \geq 0.9(v-u) + c_7 (\mu^{-1}\eta)^k
\end{equation}
in order to obtain \eqref{faster}.
Then, in the proof above, we replace $d-u$ with $v-u$ and $s-Y_{N,k}$ with 0 in the constraints \eqref{constraint1}--\eqref{constraint5} and in \eqref{qxheight}. Otherwise the proof remains the same, establishing \eqref{Gammaslow2}.
In the end, in all the Cases 1--3 we have the contradiction \eqref{faster}, and Lemma \ref{lem:geodesics} is proven.
\end{proof}

\subsection{Conclusion of the proof of Theorem \ref{main}}
\label{sec:i-iii}

Item (ii).
This follows from the combination of Lemma~\ref{lem:success} and Lemma~\ref{lem:geodesics}.
For a configuration $\omega$ let $n_1(\omega)<n_2(\omega)<\dots$ be the indices $k$ for which $\omega\in F_k$.  
Let $\Gamma_0$ be an infinite geodesic starting from the origin, with sites $0=x_0,x_1,\dots$.  
Then Lemma \ref{lem:geodesics} says that for each $j\geq 1$, $\Gamma_0$ is contained in either $\Lambda_{H,n_j}$ or $\Lambda_{V,n_j}$ until it leaves $\gO_{n_j}^{NE}\cap\gO_{n_j}^{NW}$; accordingly, we say $\Gamma_0$ is {\it horizontal at stage~j} or {\it vertical at stage~j}.

Suppose $\Gamma_0$ is horizontal at stage~$j$, and vertical at stage~$j+1$. The horizontal coordinate of the first point of $\Gamma_0$ outside $\gO_{n_j}^{NE}\cap\gO_{n_j}^{NW}$ then has magnitude at least $c/r_{n_j} - \mu^{-n_j}$ but at most $\mu^{-n_{j+1}}$  (since $\Gamma_0$ is vertical at stage $j+1$.)  Since $ \theta < \mu$ by \eqref{params2}, this means that for large $j$ we have $c/(2r_{n_j}) \le \mu^{-n_{j+1}}$.
Choosing $\gep>0$ small enough so that $\mu^{1+\gep}>\theta$, thanks to \eqref{nj} we get that for large $j$, $n_{j+1} \le (1+\gep)   n_j$, so that
\[
\Big(\frac{\mu^{1+\gep} }{\theta} \Big)^{n_j} \leq \frac{\mu^{n_{j+1}}}{r_{n_j}} \leq \frac{2}{c}.
\]
But this can only be true for finitely many $j$. Hence there exists a random $J_0$ such that for $j\geq J_0$, either $\Gamma_0$ is horizontal at stage $j$ for all $j\geq J_0$, or $\Gamma_0$ is vertical at stage $j$ for all $j\geq J_0$.   (We call $\Gamma_0$ {\it horizontal} or {\it vertical}, accordingly.) Using again that $n_{j+1}/n_j \to 1$, we get that
\[
\mu^{-n_{j+1}} \ll \frac{c}{ r_{n_j} } \quad\text{as } j\to\infty,
\]
that is, the width of $\Lambda_{*,n_{j+1}}$ is much less than the length of $\Lambda_{*,n_j}$, for $*=$ H, V. This guarantees that for such $\Gamma_0$, the angle to $x_i$ from an axis approaches 0, that is, $\Gamma_0$ is directed in an axis direction, proving Theorem~\ref{main}(ii).

\begin{remark}\rm 
The above reasoning gives that geodesics reaching horizontal distance $n= c/r_{k}= c' \theta^{-k}$ deviate from the horizontal axis by at most $\mu^{-k}$. 
In the other direction, heuristically, in order to reach horizontal distance $n = c/r_{k}$, the most efficient way should involve a route going as soon as possible (through a succession of horizontal and zigzag highways) to the closest horizontal highway that reaches at least distance $c/r_k$, and then following that highway. This suggests that in order to reach distance $n=c/r_k$, geodesics have a transversal fluctuation of order at least $1/\tilde r_{q_k} = c \tilde \theta^{- c_\theta  k}=c\theta^{-c_{\tgt}k}$, or equivalently order $n^{c_{\tgt}}$, as this is the typical vertical distance to the closest horizontal highway reaching $n$. 

Suppose $c_{\tgt}>10/11$.  Then choosing $c_\theta$ slightly less than 0.5, then $\eta$ slightly less than 2/3, and then $\delta$ sufficiently small, we satisfy \eqref{params}, \eqref{params2}, and the conditions preceding them, and further, $\tgt^{c_\theta}$ is the smallest of the 3 quantities on the right side of \eqref{params2}.  This means we can choose $\mu$ arbitrarily close to $\tgt^{c_\theta}=\theta^{c_{\tgt}}$.  Thus our upper bound of $\mu^{-k}$ becomes $n^{c_{\tgt} +o(1)}$, nearly matching the heuristic lower bound.  
Hence in this case we expect geodesics reaching horizontal distance $n$ to have transversal fluctuations of order $n^{c_{\tgt} +o(1)}$, with at least all values $c_{\tgt}\in (10/11,1)$ being possible.
\end{remark}

\smallskip
Item (i). 
It follows readily from an upper bound in the same style as the lower bound \eqref{IkMk} that for $\epsilon>0$, $\bP(X_{i,k}^{\ast} \geq \epsilon 2^k)$ is summable over $k$, so that $X_{i,k}^{\ast} =o(2^k)$ a.s.~(meaning there are long diagonal highways close to the origin), and therefore  the asymptotic speed in any diagonal direction is $\sqrt{2}/1.4$, and similarly along an axis it is $1/0.9$.  

Observe that \eqref{TGlower} in the proof of Lemma \ref{lem:geodesics} is valid for all geodesics, as that is the only property of $\Gamma_{0x}(a_x,b_x)$ that is used.  Consider then the geodesic from $(0,0)$ (in place of $a_x$) to some point $(r,s)$ with $0\leq s\leq r$ (in place of $b_x$.)  It is easy to see that the right side of \eqref{TGlower} is not increased if we replace the geodesic with a path of $2s$ consecutive zigzag bonds (heading NE) and $r-s$ horizontal bonds heading east---effectively this means consolidating all zigzag bonds into a single highway. It follows that the passage time is 
\[
  \tau((0,0),(r,s)) \geq \left(.9(r-s) + 1.4s\right)(1+o(1)) \quad \text{as } b\to\infty.
\]
From the fact that there are, with high probability, long HV and zigzag highways close to any given point (reflected in the fact that $X_{*,k}=o(2^k)$ and $Y_{*,k}=o(2^k)$), we readily obtain the reverse inequality.  The linearity of the asymptotic expression $.9(r-s) + 1.4s$ means that the limit shape is flat between any diagonal and an adjacent axis.
It follows that the limit shape $\mB$ is an octogon, with vertex $(1/.9,0)$ on the horizontal axis, $(\sqrt{2}/1.4,\sqrt{2}/1.4)$ on the SW/NE diagonal, and symmetrically in other quadrants, proving Theorem~\ref{main}(i).

\smallskip
Item (iii).
Let $F$ be the facet of $B$ in the first quadrant between the horizontal axis and the main diagonal, 
and let $\rho_F$ be the linear functional equal to 1 on $F$. 
We have from Theorem~1.11 in \cite{DH14} that there is a semi-infinite geodesic $\Gamma_F$ with $\dir(\Gamma_F) \subset \{v/|v|:v\in F\}$, and Theorem~4.3, Corollary 4.7 and Proposition 5.1 of \cite{DH14} show that $\Gamma_F$ has Busemann function linear to $\rho_F$.
But all geodesics are directed in axis directions so we must have $\dir(\Gamma_F)=\{(1,0)\}$.
Let $\hF$ be the facet which is the mirror image of $F$ across the horizontal axis.  From lattice symmetry, we have $\dir(\Gamma_{\hF})=\{(1,0)\}$ and its Busemann function is linear to $\rho_{\hF}$.  Since the Busemann functions differ, we must have $\Gamma_F\neq\Gamma_{\hF}$.   This and lattice symmetry prove Theorem \ref{main}(iii).


\begin{thebibliography}{99}

\bibitem{AH16} Ahlberg, D. and Hoffman, C., \textit{Random coalescing geodesics in first-passage percolation}, (2016), arXiv:1609.02447 [math.PR]

\bibitem{Bo90} Boivin, D.
\textit{First passage percolation:  the stationary case}, Probab. Theory Related Fields {\bf 86} (1990), no. 4, pp.~491--499.

\bibitem{BH17} Brito, G. and Hoffman, C., personal communication.

\bibitem{DH14} Damron, M. and Hanson, J., \textit{Busemann functions and infinite geodesics in two-dimensional first-passage percolation}, Comm. Math. Phys. {\bf 325} (2014), no. 3, pp.~917--963.

\bibitem{DH17} Damron, M. and Hanson, J., \textit{Bigeodesics in first-passage percolation}, Comm. Math. Phys. {\bf 349} (2017), no. 2, pp.~753--756.

\bibitem{HM95} H\"aggstr\"om, O. and Meester, R., \textit{Asymptotic shapes for stationary first passage percolation}, Ann. Probab. {\bf 23} (1995), no.~4, pp.~1511--1522.

\bibitem{Ho08} Hoffman, C., \textit{Geodesics in first passage percolation}, Ann. Appl. Probab. {\bf 18}, (2008), no. 5, pp.~1944--1969.

\end{thebibliography}
\end{document}